\documentclass[11pt]{article}
\usepackage[T1]{fontenc}
\usepackage[utf8]{inputenc}
\usepackage[left=2cm,right=2cm,bottom=3cm,top=2cm]{geometry}
\usepackage{hyperref}
\usepackage{makecell}

\usepackage{graphicx}
\usepackage{epsfig}
\usepackage{url}
\usepackage{color}
\usepackage{trivfloat}

\usepackage{subcaption}
\usepackage{amsmath,amstext,amsbsy,amsopn,amscd,amsxtra,upref,amssymb,amsthm}
\usepackage{bm}
\usepackage{grffile}
\usepackage{algorithm}
\usepackage{algpseudocode}

\newcommand{\PP}[1]{\mathbb{P}\left\{#1\right\}}

\renewcommand{\Vec}{\operatorname{vec}}
\newcommand{\Mat}{\operatorname{mat}}

\newcommand{\R}{\mathbb{R}}

\newcommand{\Complex}{\mathbb{C}}

\newcommand{\lam}{\lambda}
\newcommand{\Lam}{\Lambda}

\newcommand{\mb}[1]{\left[\begin{array}{#1}}
\newcommand{\me}{\end{array}\right]}
\newcommand{\smb}{\left[\begin{smallmatrix}}
\newcommand{\sme}{\end{smallmatrix}\right]}

\newcommand{\kr}{\odot}
\newcommand{\kron}{\otimes}
\newcommand{\hadprod}{\ast}

\newcommand{\pinv}{\dagger}
\DeclareMathOperator{\diag}{diag}

\DeclareMathOperator{\spann}{span}

\newtheorem{theorem}{Theorem}
\newtheorem{example}[theorem]{Example}
\newtheorem{lemma}[theorem]{Lemma}
\newtheorem{remark}[theorem]{Remark}
\newtheorem{corollary}[theorem]{Corollary}
\newtheorem{definition}[theorem]{Definition}

\renewcommand{\algorithmicrequire}{\textbf{Input: }}
\renewcommand{\algorithmicensure}{\textbf{Output: }}

\DeclareFontFamily{U}{matha}{\hyphenchar\font45}
\DeclareFontShape{U}{matha}{m}{n}{
      <5> <6> <7> <8> <9> <10> gen * matha
      <10.95> matha10 <12> <14.4> <17.28> <20.74> <24.88> matha12
      }{}
\DeclareSymbolFont{matha}{U}{matha}{m}{n}
\DeclareFontSubstitution{U}{matha}{m}{n}

\DeclareFontFamily{U}{mathx}{\hyphenchar\font45}
\DeclareFontShape{U}{mathx}{m}{n}{
      <5> <6> <7> <8> <9> <10>
      <10.95> <12> <14.4> <17.28> <20.74> <24.88>
      mathx10
      }{}
\DeclareSymbolFont{mathx}{U}{mathx}{m}{n}
\DeclareFontSubstitution{U}{mathx}{m}{n}

\DeclareMathDelimiter{\vvvert}{0}{matha}{"7E}{mathx}{"17}

\begin{document}

\title{Subspace embedding with random Khatri--Rao products \\ and its application to eigensolvers}

\author{%
Zvonimir Bujanovi\'c\thanks{University of Zagreb, Faculty of Science, Department of Mathematics, Croatia. E-mails: \texttt{zbujanov@math.hr}, \texttt{luka@math.hr}. The work of both authors was supported by Hrvatska Zaklada za Znanost (Croatian Science Foundation) under the grant IP-2019-04-6268 - Randomized low-rank algorithms and applications to parameter dependent problems.} \and
Luka Grubi\v{s}i\'c\footnotemark[1] \and
Daniel Kressner\thanks{\'Ecole Polytechnique F\'ed\'erale de Lausanne (EPFL), Institute of Mathematics, 1015 Lausanne, Switzerland. E-mails: \texttt{daniel.kressner@epfl.ch}, \texttt{hysan.lam@epfl.ch}. The work of both authors was supported by the SNSF research project \textit{Fast algorithms from low-rank updates}, grant number: 200020\_178806.}\and Hei Yin Lam\footnotemark[2]}

\maketitle

\begin{abstract}
Various iterative eigenvalue solvers have been developed to compute parts of the spectrum for a large sparse matrix, including 
the power method, Krylov subspace methods, contour integral methods, and preconditioned solvers such as the so called LOBPCG method.
All of these solvers rely on random matrices to determine, e.g., starting vectors that have, with high probability, a non-negligible overlap with the eigenvectors of interest. For this purpose, a safe and common choice are unstructured Gaussian random matrices. In this work, we investigate the use of random Khatri--Rao products in eigenvalue solvers. On the one hand, we establish a novel subspace embedding property that provides theoretical justification for the use of such structured random matrices. On the other hand, we highlight the potential algorithmic benefits when solving eigenvalue problems with Kronecker product structure, as they arise frequently from the discretization of eigenvalue problems for differential operators on tensor product domains. In particular, we consider the use of random Khatri--Rao products within a contour integral method and LOBPCG. Numerical experiments indicate that the gains for the contour integral method strongly depend on the ability to efficiently and accurately solve (shifted) matrix equations with low-rank right-hand side. The flexibility of LOBPCG to directly employ preconditioners makes it easier to benefit from Khatri--Rao product structure, at the expense of having less theoretical justification.
\end{abstract}

\section{Introduction}
\label{sec:intro}
During the last decade, the significance of randomization in numerical linear algebra has been increasingly recognized; see~\cite{MartinssonTropp20,MurrayDemmel23} and the references therein. In particular, randomized sketching is used as a simple but effective technique to build a dimension reduction map (DRM): A problem that features a potentially large input matrix $B \in \R^{m \times n}$ is reduced to a smaller one by replacing $B$ with $B \Omega$, where $\Omega \in \R^{n \times \ell}$, $\ell \ll n$, is a random matrix. This idea has been very successfully used as a basis for, e.g., the randomized SVD \cite{halko2011finding}, or in subspace projection methods for large-scale eigenvalue problems, such as FEAST \cite{FEAST-Polizzi}. In fact, essentially all iterative methods for large-scale eigenvalue problems, including the power method and the Lanczos algorithm~\cite{Golub2013}, make use of random initial guesses, which effectively involves sketching.

Random matrices that are typically applied as DRMs include random Gaussian and Rademacher matrices, subsampled randomized Hadamard (SRHT) or Fourier (SRFT) transforms; see~\cite[Ch. 2]{MurrayDemmel23}. To derive probabilistic error bounds for algorithms involving such DRMs, one needs to take properties of the particular random matrix distribution into account. In particular, one requires some overlap with the fixed, but unknown subspace of interest (e.g., an invariant subspace). A popular way to phrase this is the \emph{oblivious subspace embedding} (OSE) property~\cite{Sarlos2006}, which is a generalization of the well-known Johnson-Lindenstrauss (JL) property. Given a vector  $x$, a matrix $\Omega \in \R^{n \times \ell}$ satisfies the JL property if the application of $\Omega^T$ preserves the norm of $x$ up to some prescribed relative tolerance $\varepsilon$:
$$
    (1-\varepsilon) \|x\|_2 \leq \|\Omega^T x\|_2 \leq (1+\varepsilon)\|x\|_2.
$$  
A random matrix $\Omega$ drawn from a random distribution $\mathcal{D}$ has the $(\varepsilon, \delta, k)$-OSE property if, with probability at least $1-\delta$, the JL property holds for all vectors $x$ in an arbitrary but fixed $k$-dimensional subspace $\mathcal{U} \subset \R^n$. This is is equivalent to verifying that
\begin{equation}
    \label{eq:ose}
    \PP{\|(\Omega^T U)^T (\Omega^T U) - I\|_2 > \varepsilon } < \delta, 
\end{equation}
holds for fixed but arbitrary $U \in \R^{n \times k}$ with orthonormal columns. Here, the probability is taken with respect to $\Omega \sim \mathcal{D}$.
Random Gaussian matrices $\Omega \in \R^{n \times \ell}$ satisfy the OSE property when $\ell \sim (k + \log(1/\delta))\varepsilon^{-2}$; other classes of random matrices such as SRHT/SRFT are OSE as well but with less favorable requirements for $\ell$; see \cite{Woodruff2014Sketching, MartinssonTropp20, MurrayDemmel23} and the references therein. 

\paragraph{Random Khatri--Rao matrices.} In applications, it is often beneficial to employ DRMs that exploit the underlying structure of the problem. As highlighted in~\cite[Ch. 7]{MurrayDemmel23}, one particularly useful class of structured DRMs are based on Khatri--Rao products of random matrices. Given two matrices, $\tilde{\Omega} \in \R^{\tilde{n} \times \ell}$ with columns $\tilde{\omega}_1, \ldots, \tilde{\omega}_\ell$, and $\hat{\Omega} \in \R^{\hat{n} \times \ell}$ with columns $\hat{\omega}_1, \ldots, \hat{\omega}_\ell$, their Khatri--Rao product is defined as
\begin{equation*}
    \label{eq:kr}    
    \tilde{\Omega} \kr \hat{\Omega} = \mb{cccc} \tilde{\omega}_1 \kron \hat{\omega}_1, & \tilde{\omega}_2 \kron \hat{\omega}_2, & \ldots, & \tilde{\omega}_\ell \kron \hat{\omega}_\ell \me \in \R^{n \times \ell}, \quad n = \tilde{n} \cdot \hat{n},
\end{equation*}
where the symbol $\kron$ denotes the usual Kronecker product of vectors.
One major benefit of such matrices is their compatibility with operators $A$ that themselves have Kronecker product structure. This is frequently the case in applications related to, e.g., partial differential equations (PDEs) on tensor product domains (see, e.g., \cite{Palitta21}), for which structured finite difference or finite element discretizations may lead to matrices $A$ that are short sums of Kronecker products: $A = \sum_{i=1}^{s} \tilde{A}_i \kron \hat{A}_i$. In this case, the computation of $A \Omega$  becomes much cheaper when $\Omega$ is a Khatri--Rao product: $A \Omega = \sum_{i=1}^{s} (\tilde{A}_i \tilde{\Omega}) \kron (\hat{A}_i \hat{\Omega})$. 

The first goal of this paper is to establish OSE for Khatri--Rao products of random matrices, building on recent work by Ahle et al.~\cite{ahle2020oblivious}. The requirement on the number of samples $\ell$ increases only modestly compared to unstructured random Gaussian matrices; this increase is easily overcome by the improved computational efficiency. 

\paragraph{Eigenvalue solvers with random Khatri--Rao matrices.}

The second goal of the paper is to explore the potential of using random Khatri--Rao products within eigensolvers. Contour integration methods~\cite{FEAST-Polizzi,FEAST-SubspaceIteration,Sakurai-FirstPaper,Sakurai-Block,Beyn} are particularly well suited for the task; there, computing the eigenvalues of the matrix $A$ that lie inside a contour $\Gamma \subseteq \Complex$ reduces to approximating the integral of the resolvent applied to a (random) matrix $\Omega$ by using a quadrature formula:
\begin{equation}
    \label{eq:contour}
    \frac{1}{2\pi i} \int_{\Gamma} (zI - A)^{-1} \Omega \, \textrm{d}z 
        \approx \frac{1}{2\pi i} \sum_{i=1}^q w_i (z_i I - A)^{-1} \Omega.
\end{equation}
The range of the computed matrix on the right-hand side approximately spans the corresponding invariant subspace. Assume that $\Omega$ is a Khatri--Rao product of two matrices. Then, instead of solving a sequence of shifted linear systems with the large matrix $A$, evaluating \eqref{eq:contour} can be rewritten as solving a sequence of Sylvester equations with rank-one right hand side, which can be done much more efficiently; see Section \ref{sec:contour}. In both, theory and practice, we confirm that the accuracy of the computed invariant subspace is comparable to using a random Gaussian $\Omega$ with no underlying structure. 

The LOBPCG (Locally Optimally Block Preconditioned Conjugate Gradient) method~\cite{Knyazev2001Toward} for computing extreme eigenvalues of a large symmetric positive definitive matrix can also be implemented so that it exploits an initial iteration with a Khatri--Rao product structure. We show that this is possible by keeping the subsequent iterations in a low-rank factored form, and by limiting the rank of the iterates via truncation. Experimentally, we observe that this does not hamper convergence and leads to an efficient algorithm.

\paragraph{Related work.}
Random matrices with Kronecker/tensor product structure have been studied intensively in the literature, especially in algorithms for tensor decompositions. There, various constructions such as Kronecker SRFT \cite{BBK18}, TensorSketch \cite{Pagh13} or recursive tensoring \cite{ahle2020oblivious} are made in order to preserve computational efficiency while keeping the desirable sketching properties. Early work on using random Khatri--Rao products in applications includes \cite{BBB-KhatriRao-First15,KressnerPerisa17}, without a full theoretical analysis. So far, only a few works have attempted to establish JL and OSE properties for DRMs with the precise structure of the Khatri--Rao product. For the case of random matrices $\tilde{\Omega}$, $\hat{\Omega}$ with independent sub-Gaussian columns, it was recently shown \cite[Theorem 42]{ahle2020oblivious} that the JL property holds when
\begin{equation}
    \label{eq:ahlebound}
    \ell \sim \log(1/\delta)\varepsilon^{-2} + \log(1/\delta)^2 \varepsilon^{-1},
\end{equation}
improving earlier results reported in \cite{rakhshan2020tensorized,sun2018tensor}. For a subsampled Kronecker product of random matrices that has Khatri--Rao product structure, a slightly worse result has been established in \cite{Bamberger22}.

Random Khatri--Rao matrices have also been studied in the context of trace estimation \cite{BujanovicKressner21}. A recent result \cite{meyer2023hutchinson} shows that an order-$d$ random tensor obtained by taking Khatri--Rao product of $d$ random Gaussian matrices may drop in the performance as a trace estimator exponentially in $d$, which was also indicated by the results in~\cite{Vershynin20}. In this paper we are only interested in studying the case $d=2$.

Eigenvalue solvers that make use of the Kronecker structure of the matrix $A$ are also well-known in the literature. In particular, \cite{kressner2011preconditioned} studies the LOBPCG algorithm with block size 1 where the iterates are kept in a low-rank hierarchical Tucker format, and \cite{KressnerSteinlechnerUschmajew14} uses alternating optimization combined with the tensor-train format.
In this paper, we derive a novel variant of the LOBPCG algorithm with arbitrary block size, storing the vectors in a low-rank format adapted from \cite{KressnerSteinlechnerUschmajew14}.

To the best of our knowledge, this is the first paper to apply and analyze a contour-integral based eigensolver with random Khatri--Rao matrices.

\paragraph{Structure of the paper.}

In Section \ref{sec:analysis-ose}, use~\eqref{eq:ahlebound} to derive the requirement on $\ell$ that ensures the OSE property for Khatri--Rao products of Gaussian random matrices. The original derivation of~\eqref{eq:ahlebound} addresses the more general sub-Gaussian case and does not provide explicit constants. Mainly for the convenience of the reader and because it might be of independent interest, we have derived the constants for~\eqref{eq:ahlebound} for the Gaussian case in Appendix~\ref{sec:appendix}.
In Section \ref{sec:contour} we introduce a contour integration algorithm for computing eigenvalues that uses random Khatri--Rao products, and derive a probabilistic bound on the quality of the computed approximate invariant subspace. Finally, in Section \ref{sec:lobpcg}, we develop a low-rank variant of the LOBPCG algorithm, and provide numerical evidence of its efficiency.

All algorithms were implemented in Matlab. The numerical experiments were executed on a desktop computer with Intel Core i9-9900X CPU and 64GB RAM, running Ubuntu 22.04 and Matlab R2022b. The code to reproduce the numerical experiments is publicly available\footnote{\href{https://github.com/PMF-ZNMZR/khatri-rao-embedding}{https://github.com/PMF-ZNMZR/khatri-rao-embedding}}.

\section{Oblivious subspace embedding with random Khatri--Rao products}
\label{sec:analysis-ose}
In this section, we will establish the OSE property for Khatri--Rao matrices of Gaussian random matrices, based on an existing JL property; Theorem 42 from \cite{ahle2020oblivious}.

\subsection{JL moment property}

The following definition contains the usual probabilistic form of the JL property. 

\begin{definition} \label{def:jl}
For $\delta, \varepsilon > 0$, a random $n \times \ell$ matrix $\Omega$ satisfies the $(\varepsilon,\delta)$-distributional JL property, if
$\PP{|\|\Omega^T x\|_2^2-1|>\varepsilon}\leq \delta$ holds for any $x\in \mathbb{R}^n$ with $\|x\|_2 = 1$.
\end{definition}

In many cases, including random matrices with tensor product structure, it is easier to first obtain moment bounds instead of directly establishing the tail bound of Definition~\ref{def:jl}. %

\begin{definition}
\label{def: JL moment}
    For $\delta,\epsilon> 0$ and a positive integer $p$, a random $n \times \ell$ matrix $\Omega$ satisfies the $(\varepsilon,\delta,p)$-JL moment property, if 
    \begin{equation*}
        \big(\mathbb{E}\left|\|\Omega^T x\|^2_2-1\right|^p\big)^{\frac{1}{p}}\leq \varepsilon \delta^{\frac{1}{p}}\quad\text{and}\quad \mathbb{E}\|\Omega^T  x\|_2^2=1.
    \end{equation*} 
    hold for any $x\in \R^n$ with $\|x\|_2=1$.
\end{definition}
Let us recall that Definition~\ref{def: JL moment} implies Definition~\ref{def:jl} because of Markov's inequality:
\begin{equation}
\label{eq: JL lemma}
    \PP{\big|\|\Omega^T x\|^2_2-1\big|>\varepsilon} \leq \varepsilon^{-p}\cdot \mathbb{E}\big|\|\Omega^T x\|^2_2-1\big|^p \leq \delta.
\end{equation}

The Khatri--Rao product of $d$ independent sub-Gaussian random matrices satisfies the JL-moment property, provided that $\ell$ is sufficiently large~\cite[Theorem 42]{ahle2020oblivious}. In this paper, we focus on the  special case of $d=2$ Gaussian random matrices. We therefore restate \cite[Theorem 42]{ahle2020oblivious} for this case, and also provide an explicit constant for the lower bound of $\ell$, which was not provided in~\cite{ahle2020oblivious}. We include the proof of Theorem \ref{Thm: JLmoment} in the Appendix for the convenience of the reader, tracking all the constants involved.
\begin{theorem}
\label{Thm: JLmoment}
    Let $\varepsilon\in (0,1]$, $ \delta\in (0,e^{-8}]$ and $n=\tilde{n}\hat{n}$. Choose $\Omega=\frac{1}{\sqrt{\ell}}(\tilde{\Omega} \odot \hat{\Omega})\in \mathbb{R}^{n\times \ell}$ with independent Gaussian random matrices $\tilde{\Omega}\in \mathbb{R}^{\tilde{n}\times \ell}$ and $\hat{\Omega}\in \mathbb{R}^{\hat{n}\times \ell}$. Then $\Omega$ satisfies the $(\varepsilon,\delta,p)$-JL moment property with $p=\lceil \frac{1}{2}\log(\frac{1}{\delta})\rceil$, provided that 
    \begin{equation}
        \label{eq:JLmomentbound}
        \ell  \geq C^2\log({1}/{\delta})\varepsilon^{-2}+C\log^2({1}/{\delta})\varepsilon^{-1}, \quad C=128e^4.
    \end{equation}
\end{theorem}

To compare the bound~\eqref{eq:JLmomentbound} with results previously known in the literature, it is instructive to consider the embedding of a set of vectors rather than a single vector $x$. More precisely, let $X\subseteq \mathbb{R}^n$ contain $N$ vectors of unit norm. We aim to find $\ell$ such that
\begin{equation}
    \label{eq: JL proprety for set}
    \PP{\exists x\in X \colon \left| \|\Omega^T x\|^2_2-1\right|>\varepsilon}\leq \delta
\end{equation}
holds. By applying the union bound to~\eqref{eq: JL lemma}, 
Theorem~\ref{Thm: JLmoment} implies \eqref{eq: JL proprety for set} for Khatri--Rao products of Gaussian matrices provided that 
$
    \ell \sim {\log(N/\delta)}{\varepsilon^{-2}}+\frac{\log^2(N/\delta)}{\varepsilon^{-1}}.
$
This improves the results in~\cite{rakhshan2020tensorized,sun2018tensor}, which require $\ell \sim \log^4(N/\delta) \varepsilon^{-2}$. For fixed $\delta$ and $\varepsilon$, we thus only need $\ell \sim \log^2(N)$ instead of $\ell \sim \log^4(N)$ to embed $X$. 

\subsection{Subspace embedding property for Khatri--Rao product of Gaussian matrices}

In order to turn Theorem \ref{Thm: JLmoment} into an OSE property, we will use approximate matrix multiplication, that is, conditions such that the product of sketched matrices well approximates the product of the original matrices. The following result from~\cite{Cohen2016Optimal} is central for this purpose.

\begin{theorem}
\label{Thm: OSE to AMM}
     Given an integer $d\geq 1$ and real numbers $\varepsilon \in (0,1]$, $\delta\in (0,1/2)$, let $\Omega\in \R^{n \times \ell}$ be a random matrix satisfying the $(\varepsilon/2,\delta/9^{2d},p)$-JL moment property for some $p\geq 2$. Then, for any matrices $A$ and $B$ with $n$ rows, the following bound holds:
    \begin{equation} \label{eq:approxmult}
        \PP{\left\|(\Omega^T A)^T(\Omega^T B)-A^TB\right\|_2>\varepsilon\cdot\sqrt{\left(\|A\|^2_2+{\|A\|^2_F}/{d}\right)\left(\|B\|^2_2+{\|B\|^2_F}/{d}\right)}}<\delta.
    \end{equation}
\end{theorem}
\begin{proof}
    Lemma 4 from~\cite{cohen2015optimal_arxiv} connects the JL-moment property with moments of the random variable $\|(\Omega^T U)^T(\Omega^T U)-I\|_2$ from the OSE property~\eqref{eq:ose}. In our setting, 
    as $\Omega$ satisfies the $(\varepsilon/2, \delta/9^{2d}, p)$-JL moment property, this lemma implies 
  $
      \left(\mathbb{E}\|(\Omega^T U)^T(\Omega^T U)-I\|_2^p\right)^{1/p}<\varepsilon\delta^{1/p}
  $
  for any matrix $U\in \mathbb{R}^{n\times 2d}$ with orthonormal columns.
  Any random matrix having this property satisfied~\eqref{eq:approxmult}; see~\cite[Theorem 6]{Cohen2016Optimal}.
\end{proof}

We are now in the position to establish a new OSE property for Khatri--Rao products of Gaussian matrices. 
\begin{theorem}
\label{Thm: subspace_property_Khatri-Rao}
Given an integer $k \geq 1$, real numbers $\varepsilon \in (0,1]$, $\delta\in (0,1/2)$, and $n=\tilde{n}\hat{n}$, consider $\Omega=\frac{1}{\sqrt{\ell}}(\tilde{\Omega} \odot \hat{\Omega})\in \mathbb{R}^{n \times \ell}$ for independent Gaussian random matrices $\tilde{\Omega}\in \mathbb{R}^{\tilde{n}\times \ell}$ and $\hat{\Omega}\in \mathbb{R}^{\hat{n}\times \ell}$. Then $\Omega$ has the $(\varepsilon,\delta,k)$-OSE property, provided that
\begin{align} \label{eq:ellbound}
        \ell &\geq C\cdot \big( {k^{{3}/{2}}} \varepsilon^{-2}+k\log({1}/{\delta}) \varepsilon^{-2} + k^{{1}/{2}}\log^2({1}/{\delta}) \varepsilon^{-1} \big), \quad C= (2000 e^4)^2.
    \end{align} 
\end{theorem}

\begin{proof}
Consider arbitrary $U\in \mathbb{R}^{n \times k}$ with orthonormal columns. The result will be proven by plugging the JL moment property of Theorem~\ref{Thm: JLmoment} into Theorem~\ref{Thm: OSE to AMM}  with $A = B = U$. It remains to select the parameters appropriately.

We choose $d=\big\lceil2\sqrt{k}\big\rceil$, $\tilde{\varepsilon} = \varepsilon / d$, $\tilde{\delta} = 9^{-2d} \delta$ and verify that the conditions of
Theorem~\ref{Thm: JLmoment} are satisfied for these choices.
Clearly, $0 < \tilde{\varepsilon} \leq 1$, and
$
    0 < \tilde{\delta} = 9^{-2 \lceil2\sqrt{k}\rceil} \delta \leq 9^{-4} \leq e^{-8}.
$
The following sequence of inequalities shows that the condition~\eqref{eq:ellbound} for $\delta, \varepsilon$ implies~\eqref{eq:JLmomentbound} for $\tilde \delta, \tilde \varepsilon/2$:
\begin{align*}
 & (128e^4)^2\log({1}/{\tilde{\delta}})(\tilde{\varepsilon}/2)^{-2}+128e^4\log^2({1}/{\tilde{\delta}})(\tilde{\varepsilon}/2)^{-1}  \\
 = &(256 e^4)^2 d^2 \big( 2 d \log 9 + \log(1/\delta) \big) \varepsilon^{-2}+256 e^4 d \big( 2d \log 9 + \log(1/\delta) \big)^2 \varepsilon^{-1}  \\
 \leq & (2000 e^4)^2 k^{{3}/{2}} \varepsilon^{-2} + (620 e^4)^2 k \log(1/\delta) ( \varepsilon^{-2} +  \varepsilon^{-1} )   + 620 e^4 k^{1/2} \log^2(1/\delta)  \varepsilon^{-1}  \\
 \leq & C\cdot \big( {k^{{3}/{2}}} \varepsilon^{-2}+k\log({1}/{\delta}) \varepsilon^{-2} + k^{{1}/{2}}\log^2({1}/{\delta}) \varepsilon^{-1} \big),
\end{align*}
where the first inequality uses simple comparisons like $2\log 9 \cdot d^3 \le 60 k^{3/2}$ and the second inequality uses $\varepsilon, \delta \le 1$.
Thus, all requirements of Theorem \ref{Thm: JLmoment} are met and $\Omega$ has the $(\tilde{\varepsilon}/2, \delta/9^{2d}, p)$-JL moment property with $p = \lceil\frac{1}{2}\log(1/\tilde{\delta})\rceil$. Theorem \ref{Thm: OSE to AMM} with $A=B=U$ states that
\[
     \PP{\|(\Omega^T U)^T(\Omega^T U)-I\|_2 > \tilde{\varepsilon} \left(\|U\|^2_2+{\|U\|^2_F}/{d}\right)} < \delta.
\]
Because of $\tilde{\varepsilon} \left(\|U\|^2_2+{\|U\|^2_F}/{d}\right) = \varepsilon / d \cdot (1+{k}/{d}) \le \varepsilon$, this implies $\PP{\|(\Omega^T U)^T(\Omega^T U)-I\|_2>\varepsilon} < \delta$ and completes the proof.
\end{proof}
For fixed $\varepsilon$, Theorem \ref{Thm: subspace_property_Khatri-Rao} establishes $(\varepsilon, \delta, k)$-OSE for $\ell\sim {k^{{3}/{2}}}+k\log({1}/{\delta})+ k^{{1}/{2}}\log^2({1}/{\delta})$. In passing, we note that the straightforward combination of the usual epsilon-net argument~\cite[Section 4]{Vershynin2018High-dimensional} with the distributional JL property from Theorem~\ref{Thm: JLmoment} would lead to a significantly worse estimate: $\ell \sim k^2{\log^2(1/\delta)}+{k\log(1/\delta)}$.

Random Gaussian matrices require $\ell \sim k+\log(1/\delta)$ samples in order to be an OSE \cite{MatouOnvariants2008}. The asymptotic dependence on $k$ established by Theorem~\ref{Thm: subspace_property_Khatri-Rao} is only modestly worse, which indicates that Khatri--Rao products of Gaussian random matrices can deliver comparable performance.
This is what we also observe empirically, as demonstrated by the following example.
\begin{example} \label{example:pseudoinverse}
In the next section, we will primarily employ the result of Theorem~\ref{Thm: subspace_property_Khatri-Rao} to control $\|(\Omega^T U)^\dagger\|_2$, using that $\|(\Omega^T U)^T(\Omega^T U)-I\|_2\leq \varepsilon$ implies $\|(\Omega^T U)^\dagger\|_2 \le 1 / (1-\varepsilon)$; see, e.g.,~\cite[Lemma 5.36]{VershyninBookChapter12}.
We perform a numerical experiment to illustrate how the values of $\ell$ and $k$ affect $\|(\Omega^T U)^\dagger\|_2$ for the two cases when $\Omega$ is a Gaussian random matrix and when
$\Omega$ is a Khatri--Rao product of Gaussian random matrices. 

In the left plot of Figure \ref{fig: empirical probability}, we randomly generate a matrix $U\in \mathbb{R}^{400\times k}$ with orthonormal columns by computing a QR-factorization of a Gaussian random ${400\times k}$ matrix for $k=4,\ldots, 20$. Then, we find the smallest $\ell$ such that the empirical probability of the event $\|(\Omega^T U)^\dagger\|_2\geq 5$ is smaller than $1/50$; this is done by generating 1000 independent trials of $\Omega\in \mathbb{R}^{400\times \ell}$. The plot clearly shows that for all $k$, the number of samples $\ell$ needed by random Khatri--Rao matrices is almost the same as the number of samples needed by unstructured random matrices $\Omega$.

In the right plot, we have changed the matrix $U$ so that each of its columns is the vectorization of a rank-one matrix: we first orthogonalize an $20\times 20$ random matrix to obtain $V=[v_1,\ldots, v_{20}] \in \mathbb{R}^{20\times 20}$. Then we set $U = [u \otimes v_1, \ldots, u \otimes v_{k}]\in \mathbb{R}^{400\times k}$, where $u\in \mathbb{R}^{20}$ is a randomly generated unitary vector. We observe that the random Khatri--Rao products now require a notably larger number of samples $\ell$ compared to their Gaussian counterparts. This modest increase appears to be in line with the result of Theorem~\ref{Thm: subspace_property_Khatri-Rao}. 
The observation that random Khatri--Rao products tend to perform worse when sketching rank-one vectors has been made before in related settings~\cite{BujanovicKressner21,meyer2023hutchinson}.

\begin{figure}[h]
\includegraphics[width=0.5\textwidth]{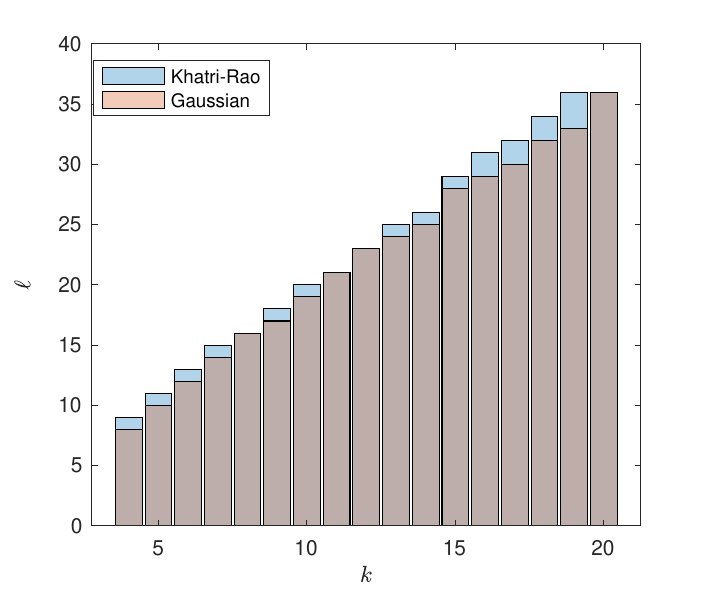}
\hspace{\fill}
\includegraphics[width=0.5\textwidth]{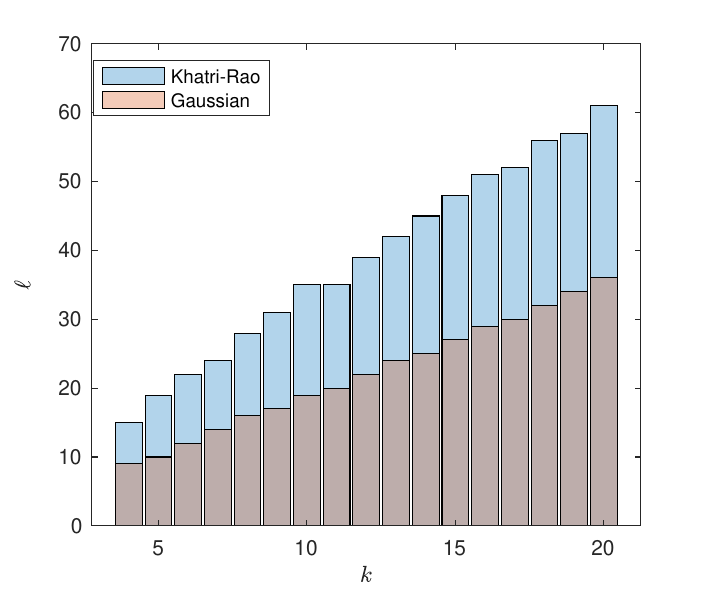}
\caption{The smallest number of samples $\ell$ in $\Omega \in \R^{400 \times \ell}$ such that the empirical probability of the event $\|(\Omega^T U)^\dagger\|_2\geq 5$ is smaller than $1/50$. Here $U \in \R^{400 \times k}$ has orthonormal columns, $k=4,\ldots, 20$. Left: $U$ is randomly generated; right: $U$ has rank-one vectors as columns.} 
\label{fig: empirical probability}
\end{figure}

In some applications, like the randomized SVD, one can cheaply mitigate the effect of a moderately large value for $\|(\Omega^T U)^\pinv\|_2$. In Figure $\ref{fig: norm of pseudo inverse}$, we therefore explore how large this norm will be if one is permitted to use a prescribed number of samples $\ell$. We fix the matrix $U \in \R^{400 \times 8}$, and generate 1000 independent random Gaussian and random Khatri--Rao product matrices $\Omega\in \mathbb{R}^{400 \times \ell}$ with $\ell=8,\ldots, 28$. The plots depict the largest, the $95$-th percentile, and the median values of $\|(\Omega^T U)^\dagger\|_2$ obtained from these trials. Once again, we have two plots that differ by the way of generating $U$ in the same way as in the Figure \ref{fig: empirical probability}. In the left plot, we observe that the Khatri--Rao products yield practically identical values as those obtained using Gaussian matrices. In the right plot, the Khatri--Rao products perform slightly worse than the Gaussian matrices. In both cases we see that $\|(\Omega^T U)^\dagger\|_2$ is bounded by a modest constant once $\ell \gtrsim 2k$.
\begin{figure}[h]
\includegraphics[width=0.5\textwidth]{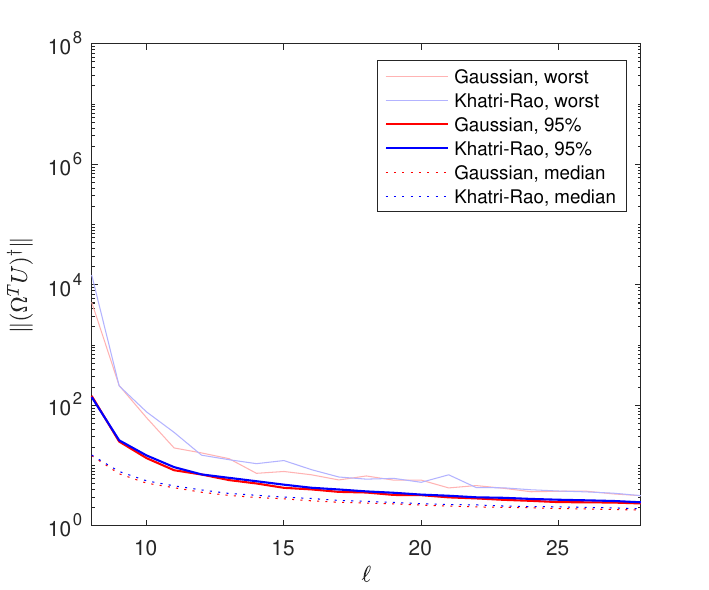}
\hspace{\fill}
\includegraphics[width=0.5\textwidth]{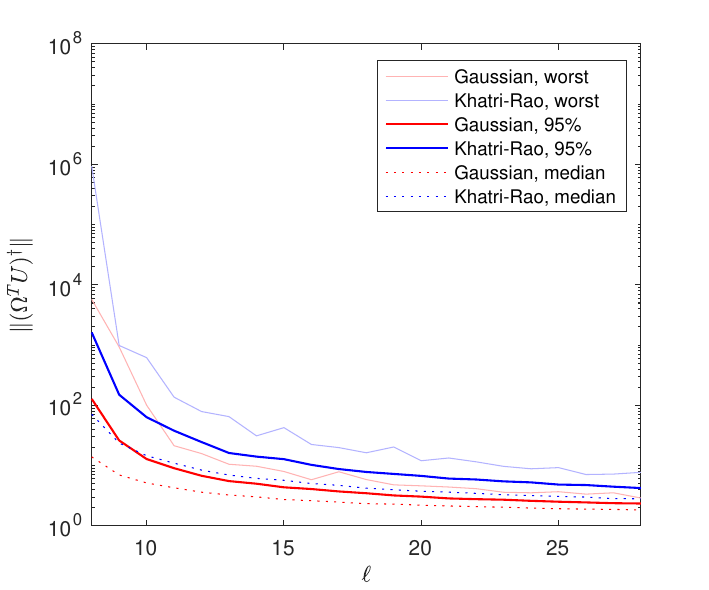}
\caption{
Statistics for $\|(\Omega^T U)^\dagger\|_2$ in $1000$ trials with a random matrix $\Omega \in \R^{400 \times \ell}$. Here $U \in \R^{400 \times 8}$ is a fixed matrix with orthonormal columns. Left: $U$ is randomly generated; right: $U$ has rank-one vectors as columns.}
\label{fig: norm of pseudo inverse}
\end{figure}

\end{example}

\section{Application of random Khatri--Rao matrices to contour integral eigensolvers}
\label{sec:contour}
In this section, we apply random Khatri--Rao product matrices to the solution of large-scale generalized eigenvalue problems. More specifically, for symmetric matrices  $A,B \in \R^{n \times n}$ with $B$ positive definite we consider the problem of computing all eigenvalues of the matrix pencil $(A, B)$ that lie inside a given contour $\Gamma \subseteq \Complex$, along with the corresponding eigenvectors. It is well known that the spectral projector $P_\Gamma$ associated with these eigenvalues can be represented by the contour integral
$
    P_\Gamma = \frac{1}{2\pi \mathrm{i}} \int_\Gamma (zB - A)^{-1}B \, \textrm{d}z.
$
Hence, for generic $\overline{\Omega} \in \R^{n \times \ell}$, the subspace spanned by the columns of
\begin{equation} \label{eq:integral}
    P_{\Gamma} \overline{\Omega} = \frac{1}{2\pi \mathrm{i}} \int_\Gamma (zI - A)^{-1} B\overline{\Omega} \, \mathrm{d}z
\end{equation}
coincides with the invariant subspace spanned by the desired eigenvectors, provided that
$\ell \ge k$, where $k$ is the number of eigenvalues inside $\Gamma$. 

As explained in the introduction, a number of eigensolvers are based on approximating the integral in~\eqref{eq:integral} by numerical quadrature, such as the trapezoidal. This effectively replaces the projector with a rational filter: $P_{\Gamma} \approx \rho(B^{-1}A)$ where the weights $w_i$ and poles $z_i$ of the filter function $\rho(z) = \frac{1}{2 \pi \mathrm{i}} \sum_{i=1}^q \frac{w_i}{z_i - z}$ are determined by the quadrature rule. Equivalently, $\rho$
can be viewed as an approximation of the indicator function on the interior of $\Gamma$~\cite{FEAST-SubspaceIteration}.
Setting $\Omega = B^{1/2} \overline{\Omega}$, we thus arrive at the approximation
\begin{equation}
    \label{eq:quadrature}    
    P_{\Gamma} \overline{\Omega} = P_{\Gamma} B^{-1/2}\Omega \approx \rho(B^{-1}A) B^{-1/2}\Omega 
    = \frac{1}{2\pi \mathrm{i}}\sum_{i=1}^q w_i (z_i B - A)^{-1} B^{1/2}\Omega 
\end{equation}
The evaluation of $P_{\Gamma} \overline{\Omega}$ thus reduces to solving $q$ shifted linear systems $(z_i B - A)^{-1} B^{1/2}\Omega$.
To mitigate the quadrature error and improve accuracy, the popular FEAST method~\cite{FEAST-Polizzi,FEAST-SubspaceIteration} applies $\rho(B^{-1}A)$ repeatedly in a subspace iteration, but we will not make use of this technique in order to fully exploit the structure we impose on $\Omega$.

Suppose that the eigenvalues $\lam_1, \ldots, \lam_n$ of $(A, B)$ are ordered such that
$$
    |\rho(\lam_1)| \geq \ldots \geq |\rho(\lam_{k})| \geq |\rho(\lam_{k+1})| \geq \ldots \geq |\rho(\lam_n)|.
$$
For a sufficiently good filter $\rho$, we expect that $\lambda_1,\ldots, \lambda_k$ are the eigenvalues inside $\Gamma$ and $|\rho(\lam_{k})| \gg |\rho(\lam_{k+1})|$. Also, we expect that the eigenvectors associated with these eigenvalues are nearly contained in the subspace spanned by the columns of $\rho(B^{-1}A)B^{-1/2}\Omega$.
Theorem~\ref{thm:structural} below provides bounds that justify such statements. It also shows how the choice of $\Omega$ affects the quality of the obtained approximation. Note that we use angles in the scalar product induced by $B$: For nonzero vectors $u$, $z$, we decompose $z = u + u_\perp$ such that $u^T B u_\perp = 0$ and set $\tan \angle_B (u, z) := \|u_\perp\|_B / \|u\|_B$, where $\|u\|_B = \sqrt{u^T B u}$.
For a subspace $\mathcal{Z}$, we define $\tan \angle_B (u, \mathcal{Z}) := \min_{z \in \mathcal{Z}} \tan \angle_B (u, z)$. 
Similar results, focusing mainly on the role of the filter are well known in the literature, see, e.g.,  \cite[Theorem 2.2]{Guettel2015} in the context of the FEAST method.
\begin{theorem} 
    \label{thm:structural}
    With the notation introduced above, let us assume that $\ell \geq k$ and $|\rho(\lam_k)| > |\rho(\lam_{k+1})|$. Let $u_1, \ldots, u_n$ denote $B$-orthonormal eigenvectors corresponding to $\lam_1, \ldots, \lam_n$, respectively, and let $U = B^{1/2} [u_1, \ldots, u_k]$ and $U_{\perp} = B^{1/2} [u_{k+1}, \ldots, u_n]$.
    If the matrix $U^T \Omega \in \R^{k \times \ell}$ has full row rank, then
    \begin{equation}
        \label{eq:tm-structural-tan}
        \tan \angle_B( u_j, \spann(Z)) 
            \leq \| \rho(\Lam_{\perp}) (U_\perp^T \Omega) (U^T \Omega)^\pinv e_j\|_2 / |\rho(\lambda_j)|, \quad  1 \leq j \leq k,
    \end{equation}
    with $Z = \rho(B^{-1}A) B^{-1/2} \Omega$. Here $\Lam_{\perp} = \diag(\lam_{k+1}, \ldots, \lam_n)$, and  $e_j \in \R^k$ is the $j$th unit vector.
\end{theorem}
\begin{proof}
    Let $\Lam = \diag(\lam_1, \ldots, \lam_k)$. Since $\rho(B^{-1}A)u_j = \rho(\lam_j)u_j$, we have 
    \begin{equation}
        \label{eq:tm-structural-rhoBV}
        \rho(B^{-1}A)B^{-1/2} \mb{cc} U & U_\perp \me
            = B^{-1/2} \mb{cc} U & U_\perp \me 
                \mb{cc} 
                    \rho(\Lam) & \\ 
                    & \rho(\Lam_\perp) 
                \me.
    \end{equation}
    The columns of $[\ U\ U_\perp\ ]$ %
    form an orthonormal basis of $\R^n$, and multiplying \eqref{eq:tm-structural-rhoBV} with $[\ U\ U_\perp\ ]^T \Omega$ from the right leads to
    \begin{equation}
        \label{eq:tm-structural-Z}
        Z = B^{-1/2} U \rho(\Lam) U^T \Omega + B^{-1/2} U_{\perp} \rho(\Lam_{\perp}) U_\perp^T \Omega.
    \end{equation}
    Using that $U^T \Omega$ is of full row rank and multiplying~\eqref{eq:tm-structural-Z} by $(U^T \Omega)^\pinv \rho(\Lam)^{-1} e_j$ from the right implies
    $$
        \spann(Z) \ni z 
            := u_j + \frac{1}{\rho(\lam_j)} B^{-1/2}U_{\perp} \rho(\Lam_{\perp}) (U_\perp^T \Omega) (U^T \Omega)^\pinv e_j.
    $$
    Since $u_j$ is $B$-orthogonal to $B^{-1/2}U_{\perp}$ and $B^{-1/2}U_{\perp}$ has $B$-orthonormal columns, it follows that
    $$
        \tan \angle_B( u_j, z ) 
            = \| B^{-1/2} U_{\perp} \rho(\Lam_{\perp}) (U_\perp^T \Omega) (U^T \Omega)^\pinv e_j\|_B / |\rho(\lambda_j)| 
            = \|\rho(\Lam_{\perp}) (U_\perp^T \Omega) (U^T \Omega)^\pinv e_j\|_2 / |\rho(\lambda_j)|,
    $$ 
    concluding the proof.
\end{proof}

The role of $\rho$ and $\Omega$ becomes more obvious by further bounding \eqref{eq:tm-structural-tan}:
\begin{equation}
    \label{eq:tm-structural-tan-split}
    \tan \angle_B( u_j, \spann(Z) ) 
        \leq \frac{|\rho(\lambda_{k+1})|}{|\rho(\lambda_j)|} \|U_\perp^T \Omega\|_2 \|(U^T \Omega)^\pinv\|_2
        \leq \frac{|\rho(\lambda_{k+1})|}{|\rho(\lambda_j)|} \|\Omega\|_2 \|(\Omega^T U)^\pinv\|_2.
\end{equation}
The ratio ${|\rho(\lambda_{k+1})|}/{|\rho(\lambda_j)|}$ features prominently in this bound and, by choosing
an appropriate filter $\rho$ (e.g., the one obtained from the trapezoidal rule), it decreases exponentially as the number of quadrature nodes $q$ increases~\cite{TrefethenWeideman14}.

We will now discuss the effect of choosing a random Khatri--Rao product $\Omega = \tilde{\Omega} \kr \hat{\Omega}$ on the bound~\eqref{eq:tm-structural-tan-split}.  This is simple for the norm of $\Omega$:
$$
    \|\Omega\|_2^2
        = \|\tilde{\Omega} \kr \hat{\Omega}\|_2^2
        = \|(\tilde{\Omega} \kr \hat{\Omega})^T (\tilde{\Omega} \kr \hat{\Omega})\|_2
        = \|(\tilde{\Omega}^T \tilde{\Omega}) \hadprod (\hat{\Omega}^T \hat{\Omega})\|_2
        \leq \|\tilde{\Omega}^T \tilde{\Omega}\|_2 \|\hat{\Omega}^T \hat{\Omega}\|_2
        = \|\tilde{\Omega}\|_2^2 \|\hat{\Omega}\|_2^2,
$$
where the symbol $\hadprod$ denotes the Hadamard product of two matrices. Well-known results~\cite[Corollary 5.35]{VershyninBookChapter12} on norms of Gaussian random matrices state that
$$
    \|\tilde{\Omega}\|_2 \leq \sqrt{\tilde{n}} + \sqrt{\ell} + t, \quad
    \|\hat{\Omega}\|_2 \leq \sqrt{\hat{n}} + \sqrt{\ell} + t,  
$$
hold with probability at least $1 - 2 \exp(-t^2)$. The potentially large matrix size appearing in this bound can be easily mitigated by taking a few more quadrature nodes. For bounding $\|(\Omega^T U)^\pinv\|_2$, we will use Theorem \ref{Thm: subspace_property_Khatri-Rao} to conclude the following result.

\begin{corollary} \label{corollary:contour}
In the setting of Theorem~\ref{thm:structural}, let $\Omega = \tilde{\Omega} \kr \hat{\Omega}$ for independent Gaussian random matrices $\tilde{\Omega} \in \R^{\tilde n \times \ell}$, $\hat{\Omega} \in \R^{\hat n \times \ell}$ with $\ell \geq C(4k^{3/2} + 24k + 72 k^{1/2})$, where $C$ is the constant from Theorem \ref{Thm: subspace_property_Khatri-Rao}. 
    Assume that $\sqrt{\ell} + 4 \leq \min\{\sqrt{\tilde{n}}, \sqrt{\hat{n}}\}$, and that the filter $\rho$ is such that $|\rho(\lam_{k+1})|/|\rho(\lam_k)| < \varepsilon \sqrt{\ell / 64n}$ for some $0 < \varepsilon < 1$.
    Then, for all $1 \leq j \leq k$,
    $$
        \PP{ \tan \angle_B( u_j, \spann(Z)) \leq \varepsilon} > 0.997.
    $$
\end{corollary}
\begin{proof}
Note that
    $
        \PP{ \|\hat{\Omega}\|_2 > 2\sqrt{\hat{n}} } 
            \leq \PP{ \|\hat{\Omega}\|_2 > \sqrt{\hat{n}} + \sqrt{\ell} + 4 } 
            \leq 2e^{-4^2} < 10^{-6}
    $
    and an analogous bound holds for $\|\tilde{\Omega}\|_2$. 
    Using Theorem \ref{Thm: subspace_property_Khatri-Rao} with 
    $\overline{\Omega} = \frac{1}{\sqrt{\ell}} \Omega = \frac{1}{\sqrt{\ell}} \tilde{\Omega} \kr \hat{\Omega}$,
    $\varepsilon=1/2$, and $\delta = e^{-6}$, we obtain
    $$
        \mathbb P\{ \|(\Omega^T U)^\pinv\|_2 > 2/\sqrt{\ell} \}
            = \mathbb P\{ \|(\overline{\Omega}^T U)^\pinv\|_2 > 2 \}
            \leq \mathbb P\{ \| (\overline{\Omega}^T U)^T (\overline{\Omega}^T U) - I \|_2 > 1/2 \}
            \leq e^{-6}
            < 0.0024.
    $$
Using the union bound, we thus have that $\mathbb P\{ \|\Omega\|_2 \|(\Omega^T U)^\pinv\|_2 \leq 8\sqrt{n/\ell} \}$ is bounded from below by 
    \[ 1 
                - \mathbb P\{ \|\tilde{\Omega}\|_2 > 2\sqrt{\tilde{n}} \}
                - \mathbb P\{ \|\hat{\Omega}\|_2 > 2\sqrt{\hat{n}} \}
                - \mathbb P\{ \|(\Omega^T U)^\pinv\|_2 > 2/\sqrt{\ell} \} 
            > 0.997.
\]
Combined with the bound~\eqref{eq:tm-structural-tan-split} from Theorem \ref{thm:structural}, this gives
    \begin{align*}
        \PP{ \tan \angle_B( u_j, \spann(Z) ) \leq \varepsilon }
            &\geq 
                \PP{ \frac{|\rho(\lambda_{k+1})|}{|\rho(\lambda_j)|} \|\Omega\|_2 \|(\Omega^T U)^\pinv\|_2 \leq \varepsilon } \\
            &\geq 
                \PP{ \|\Omega\|_2 \|(\Omega^T U)^\pinv\|_2 \leq 8\sqrt{n/\ell} } 
            > 
                0.997.
    \end{align*}
\end{proof}

Corollary~\ref{corollary:contour} shows that the contour integral method using random Khatri--Rao matrices along with a good quadrature formula results in good eigenvector approximation, provided that a modest amount of oversampling is used. In view of the experiments reported in Example~\ref{example:pseudoinverse}, the constant involved in the asymptotic relation $\ell \sim  k^{3/2}$ is certainly a gross overestimate.

\begin{remark}
For a Gaussian random matrix $\Omega$, one can directly combine the structural bound of Theorem~\ref{thm:structural} (instead of the weaker bound~\eqref{eq:tm-structural-tan-split}) with the probabilistic analysis from~\cite{halko2011finding} developed in the context of the randomized SVD.
This results in a probabilistic bound analogous to Corollary~\ref{corollary:contour} for $\ell \sim k$.
Related results, based on a different structural bound, have been presented by Miedlar~\cite{MiedlarFOCM23}. %
\end{remark}

\subsection{Efficient implementation} \label{sec:implementation}

As mentioned in the introduction, one major benefit of using Khatri--Rao products is that they can be efficiently multiplied with Kronecker products. Contour integral methods involve the inversion of matrices, see~\eqref{eq:quadrature}, which makes it more difficult to exploit Khatri--Rao product structure. To demonstrate how this can be achieved, we now consider a model problem that is typical for eigenvalue problems featuring Kronecker product structure.

We consider a two-dimensional Schr\"{o}dinger equation
\begin{align*}
    -\Delta u(x, y) + V(x, y) \cdot u(x, y) &= \lam u(x, y), && (x, y) \in D = [a, b] \times [a, b], \\
    u(x, y) &= 0, && (x, y) \in \partial D,
\end{align*}
where $V(x, y)$ is a given potential. Using finite elements to discretize this equation would lead to a generalized eigenvalue problem $Ax=\lam Bx$ with matrices $A$ and $B \neq I$ both represented as short sums of Kronecker products. For simplicity, we use a finite differences discretization that yields $B=I$ and leads to a standard eigenvalue problem $Av = \lam v$ with
\begin{equation}
\label{eq:discretized_eigenvalue_problem}
    A = -(I \kron T + T \kron I) + \diag (V(x_i, y_j)) \in \R^{\hat{n} \tilde{n} \times \hat{n} \tilde{n}}.
\end{equation}
Here $T = \text{tridiag}(1, -2, 1) / h^2$, $x_i = a + h i$, $y_j = a + h j$, and $h = (b-a) / (\tilde{n}+1)$ where $\tilde{n} = \hat{n}$ is the number of discretization points for each coordinate. In order to obtain Kronecker structure for the last term in~\eqref{eq:discretized_eigenvalue_problem}, the potentially needs to represented or approximated as a sum of separable functions. In particular, if the potential has the form $V(x, y) = f(x) + f(y) \pm g(x)g(y)$, then 
\begin{equation}
\label{eq:discretized_eigenvalue_problem_kron_form}
    A = I \kron K + K \kron I + \tilde{V} \kron \hat{V},
\end{equation}
where $K = -T + \diag(f(x_i))$, $\tilde{V} = \pm \diag(g(x_i))$, $\hat{V} = \diag(g(x_i))$.

The contour integral method described above requires solving shifted linear systems of the form $(z I - A)x = \omega$ for each quadrature node $z \in \Complex$ and each column $\omega = \tilde{\omega} \kron \hat{\omega}$ of $\Omega$. A straightforward and often feasible idea is to apply a sparse direct solver~\cite{Davis06} to such a system. However, for large $A$, it will be beneficial to exploit the Kronecker structure and rewrite
\begin{equation}
    \label{eq:sylvester-unfolded}
    \left(I \kron \left(\frac{z}{2}I - K\right) + \left(\frac{z}{2}I - K\right) \kron I - \tilde{V} \kron \hat{V} \right) x = \tilde{\omega} \kron \hat{\omega}.
\end{equation}
as a matrix equation. For this purpose, we let $\Vec: 
\R^{\hat n \times \tilde n} \to \R^{\hat n \tilde n}$ denote vectorization, which stacks the columns of a matrix into a long vector. The inverse of $\Vec$ is denoted by $\Mat: \R^{\hat n \tilde n} \to \R^{\hat n \times \tilde n}$.
Letting $X = \Mat(x)$ and using that $\Mat(\tilde{\omega} \kron \hat{\omega}) = \hat{\omega} \tilde{\omega}^T$, we obtain a multiterm Sylvester equation of the form
\begin{equation}
    \label{eq:sylvester-folded}
    \left(\frac{z}{2} I - K\right) X + X\left(\frac{z}{2} I - K\right) - \hat{V} X \tilde{V} = \hat{\omega} \tilde{\omega}^T.
\end{equation}
It is often observed (see, e.g.,~\cite{BennerBreiten13}) that the solution of such an equation with rank-one 
right-hand side is again numerically low-rank. Most methods for large-scale matrix equations operate under this assumption and approximate $X$ in factored form: $X \approx \hat{X} \tilde{X}^\ast$. This clearly highlights the benefit of random Khatri--Rao matrices; for an unstructured random matrix $\Omega$, both the right-hand side of~\eqref{eq:sylvester-folded} and the solution $X$ are numerically full rank; see also Figure \ref{fig:feast_sylvester_solution_decay}.

To solve \eqref{eq:sylvester-folded}, we have adapted the preconditioned low-rank BiCGstab algorithm from \cite{BennerBreiten13} to multiterm Sylvester equations. The preconditioner consists of applying a few iterations of the ADI method (implemented as described in~\cite{Kuerschner16}) to the Sylvester matrix equation
$$
    \left(\frac{z}{2} I - K\right) X + X\left(\frac{z}{2} I - K\right) = \hat{P} \tilde{P}^\ast,
$$
for the low-rank right-hand sides appearing in the course of BiCGstab.
The whole procedure is summarized in Algorithm \ref{alg:feast}.

\begin{algorithm}
    \caption{Contour integration for eigenvalues of $A = I \kron K + K \kron I + \tilde{V} \kron \hat{V}$ inside a contour $\Gamma$.}
    \label{alg:feast}
    \algorithmicrequire Matrices $K, \hat{V}, \tilde{V} \in \mathbb{R}^{\tilde{n} \times \tilde{n}}$; nodes $z_i$ and weights $w_i$ of a quadrature formula for contour $\Gamma$, $i=1, \ldots, q$; upper estimate $\ell$ on number of eigenvalues inside $\Gamma$.\\
    \algorithmicensure Approximations $(\lam_i, u_i)$ to eigenpairs of $A$ for which $\lam_i$ is inside $\Gamma$.
    \begin{algorithmic}[1]
        \State Generate random Gaussian matrices $\tilde{\Omega} \in \R^{\tilde{n} \times \ell}$, $\hat{\Omega} \in \R^{\tilde {n} \times \ell}$ 
        \State Initialize $\tilde{X}_j$, $\hat{X}_j$ as empty matrices,  $j=1, \ldots, \ell$.
        \For{$i=1, 2, \ldots, q$} 
            \For{$j=1, 2, \ldots, \ell$} 
                \State Let $\tilde{\omega}_j=\tilde{\Omega}(:, j)$, $\hat{\omega}_j=\hat{\Omega}(:, j)$.
                \State Solve multiterm Sylvester eqn $(\frac{z_i}{2}I - K)X + X(\frac{z_i}{2}I - K) - \hat{V} X \tilde{V} = \hat{\omega}_j \tilde{\omega}_j^\ast$ for $X = \hat{X}_{i, j}\tilde{X}_{i, j}^\ast$.
                \State Expand $\tilde{X}_{j} = [\tilde{X}_{j}, \, \frac{w_i}{2\pi \mathrm{i}} \tilde{X}_{ij}]$, $\hat{X}_{j} = [\hat{X}_{j}, \, \hat{X}_{ij}]$.
                \Comment{Note: Optional recompression.} \label{alg:contour:expand}
            \EndFor
        \EndFor
        \State Compute the matrices $X_j = \hat{X}_j \tilde{X}_j^\ast \in \R^{\hat{n} \times \tilde{n}}$ and unfold to vectors $x_j = \Vec(X_j) \in \R^n$, $j = 1, \ldots, \ell$.
        \State Compute and return Ritz pairs ($\lam_i, u_i)$ of $A$ from the subspace $\spann\{x_1, \ldots, x_\ell\}$, see Remark \ref{rem:RitzPairs}.
    \end{algorithmic}
\end{algorithm}

\begin{remark}
    \label{rem:RitzPairs}  
    Algorithm \ref{alg:feast} applies the quadrature formula \eqref{eq:quadrature} to each column of $\Omega \in \R^{n \times \ell}$ separately, storing the results as factored matrices $X_j = \hat{X}_j \tilde{X}_j^\ast$, $j=1, \ldots, \ell$. 

    In Line \ref{alg:contour:expand} we update the partial sum with the term for the next quadrature node; to control the number of columns in $\tilde{X}_j$, $\hat{X}_j$, we may do a compression of these factors by computing a truncated SVD of $\tilde{X}_j \hat{X}_j^\ast$. This can be done efficiently by first computing QR decompositions of $\tilde{X}_j$ and $\hat{X}_j$.
    
    To obtain approximate eigenpairs, we have to compute Ritz pairs from the subspace spanned by the vectors $x_1, \ldots, x_\ell$ with $x_j = \Vec(X_j)$. 
    Special care needs to be taken so that this computation involves efficient matrix-vector products $Ax$ that respect the low-rank structure of $x$. 
    This can be done by refactoring $X_j$ as $X_j = U\Sigma_j V^\ast$, where $U$ and $V$ are common for all $j$. This block low-rank format and operations with it will be described in detail in Section \ref{sec:lobpcg}.
\end{remark}

In the following, we illustrate the algorithm introduced above with a numerical example.
\begin{example}
    \label{ex:contour}
    We consider the domain $[a,b] \times [a,b] = [-1,1] \times [-1,1]$ and the potential 
    $
        V(x, y) = (x^2 + y^2 - xy)/2.
    $
    and the goal is to compute the $4$ smallest eigenvalues of the matrix $A$; these eigenvalues lie inside the circle $\Gamma$ centered at $12.606$ with radius $9$. 
    
    We generated the matrix $\Omega$ as a Khatri--Rao product of Gaussian random matrices with $\ell = 6$ columns. The integral \eqref{eq:quadrature} is approximated by the trapezoidal rule with $40$ nodes.
    
    When the number of discretization points is set to $\tilde{n} = \hat{n} = 300$, the matrix $A$ is of size $90000 \times 90000$. 
    First, we compute the four smallest eigenvalues with {\tt eigs} in Matlab, matricize each of the associated eigenvectors into four $300 \times 300$ matrices and compute their singular values. Figure~ \ref{fig:feast_eigvec_decay} shows that the singular values drop rapidly, and that each of the eigenvectors is well approximated by a low-rank vector. Moreover, the image of all these matrices is well contained in a common low-dimensional subspace. This justifies the expectation that the final rank of the matrices $\tilde{X}_j$ and $\hat{X}_j$ in Algorithm~\ref{alg:feast} will be similarly low.

    Figure~\ref{fig:feast_sylvester_solution_decay} shows the impact on the singular values of the solution to the multiterm Sylvester equation \eqref{eq:sylvester-folded} with $z=12.606 + 9 e^{i \pi/4}$, when choosing a random matrix with Khatri--Rao structure vs. an unstructured random matrix.  Clearly, the former has rapidly decaying singular values and can be efficiently stored in a low-rank factored format, while the latter does not offer such benefits.
      
    The benefit of using a specialized low-rank solver becomes more apparant as the size increases:
    \begin{center}
        \renewcommand{\arraystretch}{1.2}
        \begin{tabular}{|c|c|c||c|c|c|c|}
            \hline 
            $\tilde{n}=\hat{n}$ & $1000$ & $2000$ & $3000$ & \makecell{worst \\[-0.1cm] $\|Ax-\lam x\|$} & \makecell{worst \\[-0.1cm] $|\lam_i - \tilde{\lam}_i|$} & memory \\ \hline \hline
            sparse direct           & $8.35$s  & $46.47$s & $138.59$s & $1 \cdot 10^{-7}$ & $6 \cdot 10^{-10}$ & $27.9$ GB \\ \hline
            BiCGstab tol=$10^{-6}$  & $5.86$s  & $12.81$s & $19.59$s  & $7 \cdot 10^{-3}$ & $1 \cdot 10^{-8}$  & $400$ MB\\ \hline
            BiCGstab tol=$10^{-10}$ & $15.85$s & $30.97$s & $56.58$s  & $6 \cdot 10^{-7}$ & $6 \cdot 10^{-10}$ & $400$ MB \\ \hline
        \end{tabular}
    \end{center}
    For varying number of discretization points $\tilde{n}=\hat{n} \in \{1000, 2000, 3000\}$ we measured the average time per quadrature node needed for solving the linear system~\eqref{eq:sylvester-unfolded}. The first row refers to the sparse direct solver utilized by Matlab's backslash operator, the second row refers to our adapted BiCGstab algorithm from~\cite{BennerBreiten13} stopping once the residual norm falls below $10^{-6}$; in the third row with the stricter stopping criterion ($10^{-10}$). 
    Using any of these solvers, we were able to find $4$ eigenvalue approximations inside the contour $\Gamma$. The table also shows the largest of all eigenpair residuals among the computed approximations, the largest of all errors in the computed eigenvalues, and the total memory used in the case $\tilde{n} = \hat{n} = 3000$. As the problem size increases, so does the benefit of using a specialized solver: both time and especially memory usage become significantly smaller than with backslash. When the BiCGstab tolerance in Algorithm \ref{alg:feast} was set to $10^{-6}$, the matrices $\tilde{X}_j$ and $\hat{X}_j$ that store the accumulated sum constantly had $29$ columns, up to the last iteration of the for-loop when the number of columns dropped to $6$ (recompression was performed during each expansion). The ADI preconditioner was configured to run at most $55$ iterations, stopping if the relative residual drops below $10^{-5}$. The maximum rank of the iteration matrices within BiCGstab was limited to $90$.

    \begin{figure}
        \centering
        \begin{subfigure}{0.48\textwidth}
            \includegraphics[width=\textwidth]{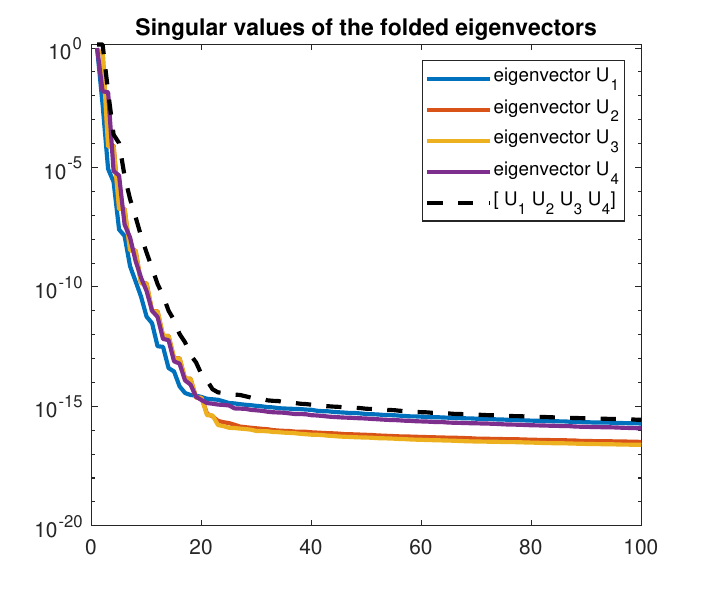}
            \caption{Singular value decay of matricized eigenvectors belonging to $4$ smallest eigenvalues.}
            \label{fig:feast_eigvec_decay}
        \end{subfigure}
        \hfill
        \begin{subfigure}{0.48\textwidth}
            \includegraphics[width=\textwidth]{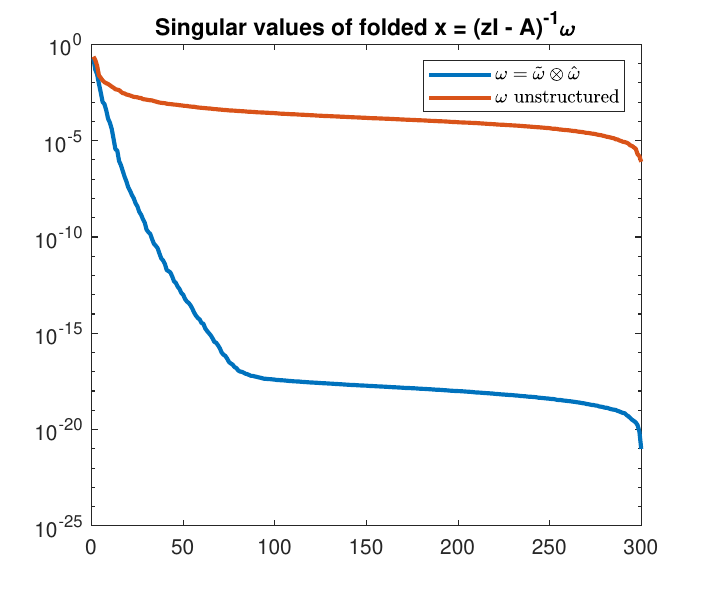}
            \caption{Singular value decay of solution to multiterm Sylvester equation \eqref{eq:sylvester-unfolded} for different right-hand sides $\omega$.}
            \label{fig:feast_sylvester_solution_decay}
        \end{subfigure}
                
        \caption{Properties of the eigenvalue problem arising from a discretized Schr\"{o}dinger equation with potential $V(x, y) = (x^2 + y^2 - xy)/2$.}
        \label{fig:figures}
    \end{figure}
\end{example}

Although the results of the numerical experiment appear to be promising; we observed at the same time that the ability of existing multiterm Sylvester methods to efficiently solve~\eqref{eq:sylvester-unfolded} is limited. In particular, while experimenting with different potentials $V$, we noticed that the BiCGstab solver often requires a very good preconditioner (that is, many ADI iterations) in order to reach convergence. Reaching convergence also sometimes failed for certain quadrature nodes, or required careful tweaking of the algorithm's parameters. Any improvement in algorithms for solving multiterm Sylvester equations would immediately lead to improvement of our proposed contour integration method for eigenvalue computation.

We also remark that Algorithm \ref{alg:feast} can be easily adapted to contour integration methods that employ higher moments, such as Beyn's method~\cite{Beyn} or the block Sakurai--Sugiura method~\cite{Sakurai-FirstPaper,Sakurai-Block}.

\section{Low-rank variant of LOBPCG}
\label{sec:lobpcg}

In view of the challenges encountered with numerically solving multiterm Sylvester equations arising in the countour integral method, we now consider a method that allows us to bypass the need for solving such equations by directly incorporating the preconditioner into the eigensolver. Among such preconditioned methods, LOBPCG~\cite{Knyazev2001Toward} is arguably the most popular one for computing extreme eigenvalues of symmetric positive definite matrices. %

LOBPCG is based on the local optimization of a three-term block recurrence. Given a symmetric positive definite matrix $A\in \mathbb{R}^{n \times n}$, the goal is to compute the $k$ smallest eigenvalues $0<\lambda_1\leq\cdots\leq \lambda_k$ together with the corresponding eigenvectors. For simplifying the description, we only consider standard eigenvalue problems (that is, $B = I_n$) in the following. The extension to general symmetric positive definite $B$ is straightforward~\cite{Knyazev2001Toward}.
 
Given a (randomly chosen) initial matrix $X^{(1)}\in \mathbb{R}^{n\times \ell}$, $\ell\geq k$, with orthonormal columns, LOBPCG computes a sequence of iterates of the form
\begin{equation*}
X^{(i+1)}=X^{(i)}C_1^{(i+1)}+R^{(i)}C^{(i+1)}_2+P^{(i)}C_3^{(i+1)},\quad \text{for } i=1,2,3,\ldots,
\end{equation*} where $R^{(i)}=M^{-1}\big(AX^{(i)}-X^{(i)}(X^{(i)})^TAX^{(i)}\big)$, and $M$ is the chosen preconditioner. The block $P^{(i)}$ is defined as through the sequence
$$P^{(i+1)}=P^{(i)}C_2^{(i+1)}+R^{(i)}C^{(i+1)}_3,$$ with $P^{(1)}$ being an empty block. Denoting $S^{(i)} := [X^{(i)}\;R^{(i)}\; P^{(i)}] \in \mathbb{R}^{n\times 3\ell}$,
the matrices $C_1^{(i+1)}, C_2^{(i+1)}, C_3^{(i+1)}$ defining these recursions are chosen as the block rows of the 
the matrix $C^{(i+1)} \in \mathbb{R}^{3\ell \times \ell}$ that minimizes 
\begin{equation}
\label{eq: min trace}
\min_{C} \big\{ \text{trace}\big((S^{(i)}C)^TAS^{(i)}C\big) \colon (S^{(i)}C)^T(S^{(i)}C)=I \big\}.
\end{equation} 
It is well known that a minimizer of \eqref{eq: min trace} is found by solving the (small) generalized eigenvalue problem for the matrix pencil $(S^{(i)})^TAS^{(i)} - \lam (S^{(i)})^TS^{(i)}$. Specifically, $C^{(i+1)}$ contains the eigenvectors belonging to the $\ell$ smallest eigenvalues. If these eigenvalues are arranged on the diagonal of the 
$\ell\times \ell$ matrix $\Theta$, we have the relation
\begin{equation}
\label{eq: generalized eigenvalues problem}
  (S^{(i)})^TAS^{(i)} C^{(i+1)}= (S^{(i)})^TS^{(i)} Y \Theta.
\end{equation} 
For a detailed discussion of numerical aspects of this procedure, see~\cite{Duersch2018robust}.  
 In particular, it is important to note that numerical instability may arise due to the ill-conditioning of $S^{(i)}$ when solving \eqref{eq: generalized eigenvalues problem}. Often, this is mitigated by orthogonalizing $S^{(i)}$ using, e.g., the Hetmaniuk-Lehoucq orthogonalization strategy~\cite{Hetmaniuk2006Basis}. 

\subsection{Exploiting low-rank and Kronecker structure}

In this section we propose a variant of LOBPCG that will exploit 
the assumed Kronecker structure of the matrix $A$ by storing the iterates in a compatible low-rank format. By doing so, we also implicitly assume that the target eigenvectors are of low-rank, as already mentioned in Example \ref{ex:contour}. This will allow for an efficient implementation of all steps of the algorithm.

The proposed algorithm is inspired by~\cite[Algorithm 2]{kressner2011preconditioned}, which discusses the case $\ell = 1$ for tensor operators $A$. In each iteration, we represent the matrices $X^{(i)}$, $P^{(i)}$ and $R^{(i)}$ in block low-rank matrix format, which will be denoted by bold face letters: $\mathbf{X}^{(i)}$, $\mathbf{P}^{(i)}$ and $\mathbf{R}^{(i)}$. This format is the matrix analogue of the format used in \cite{KressnerSteinlechnerUschmajew14} for storing multiple low-rank tensors. 
A high-level pseudocode of the proposed low-rank variant of LOBPCG in shown in Algorithm \ref{alg:lobpcg}. It remains to discuss details concering the block low-rank format.

\begin{algorithm}[h]
\caption{LOBPCG with low-rank truncation}\label{alg:lobpcg}
\algorithmicrequire Functions for applying $A$, $M^{-1}$ to vectors in block low-rank matrix format; $\ell$ starting vectors $\mathbf{X}^{(0)}$ stored in block low-rank matrix format; block low-rank matrix truncation operator $\mathcal{T}$; number $k$ of desired smallest eigenvalues.\\
\algorithmicensure  Matrix $\tilde{\mathbf{X}}$ containing converged eigenvectors, and diagonal matrix $\tilde{\Lambda}$ containing converged eigenvalues of $A$.
     \begin{algorithmic}[1]
     \State Orthonormalize $\mathbf{X}^{(0)}$: $L_X=\text{chol}\big((\mathbf{X}^{(0)})^T \mathbf{X}^{(0)}\big)$, \quad $\mathbf{X}^{(0)}=\mathbf{X}^{(0)}L_X^{-1}$. \Comment{see \eqref{eq: gram matrix} and \eqref{eq: blk matrix mutiplication}}
     \State Solve the eigenvalue problem $(\mathbf{X}^{(0)})^{T}A\mathbf{X}^{(0)}C=C\Theta$.
     \State $\mathbf{X}^{(1)}=\mathbf{X}^{(0)}C$,\quad$\mathbf{P}^{(1)}=[]$.
      \For{$i=1,2,3,\ldots,\text{(until $k$ smallest eigenvalues have converged)}$}
        \State $\mathbf{R}^{(i)}=M^{-1}\big(A\mathbf{X}^{(i)}-\mathbf{X}^{(i)}\Theta\big)$,
        \quad\quad $\mathbf{R}^{(i)}=\mathcal{T}(\mathbf{R}^{(i)}).$
        \label{alg:lobpcg:precondotioner}
         \State Orthonormalize $\mathbf{R}^{(i)}$: $L_R=\text{chol}\big((\mathbf{R}^{(i)})^T \mathbf{R}^{(i)}\big)$, \quad $\mathbf{R}^{(i)}=\mathbf{R}^{(i)}L_R^{-1}$. 
          \State Orthonormalize $\mathbf{P}^{(i)}$: $L_P=\text{chol}\big((\mathbf{P}^{(i)})^T \mathbf{P}^{(i)}\big)$, \quad $\mathbf{P}^{(i)}=\mathbf{P}^{(i)}L_P^{-1}$.
          
        \State $\mathbf{S_1}=\mathbf{X}^{(i)},\quad \mathbf{S_2}=\mathbf{R}^{(i)},\quad \mathbf{S_3}=\mathbf{P}^{(i)}$.
        \State $\Tilde{A}_{ij}=\mathbf{S_i}^T (A \mathbf{S_j}), \quad \Tilde{B}_{i,j}= \mathbf{S_i}^T\mathbf{S_j}$, \quad $i, j=1, 2, 3$. \Comment{Form the matrices in \eqref{eq: generalized eigenvalues problem}}
        \State Compute $\ell$ smallest eigenvalues and their eigenvectors: $\Tilde{A} C=\Tilde{B}C \Theta$.
        \State Partition $C=\begin{bmatrix}
C_1\\
C_2\\
C_3
\end{bmatrix}.$
        \State $\mathbf{P}^{(i+1)}=\mathbf{R}^{(i)}C_2+\mathbf{P}^{(i)}C_3$,\quad\quad $\mathbf{P}^{(i+1)}=\mathcal{T}(\mathbf{P}^{(i+1)})$.
        \State $\mathbf{X}^{(i+1)}= \mathbf{X}^{(i)}C_1+ \mathbf{P}^{(i+1)}$,\quad\quad $\mathbf{X}^{(i+1)}=\mathcal{T}(\mathbf{X}^{(i+1)})$.      
        \EndFor
    \State $\tilde{\Lam} = \Theta$, $\tilde{\mathbf{X}} = \mathbf{X}^{(i)}$.
     \end{algorithmic}
\end{algorithm}

\begin{paragraph}{Definition of block low-rank format.}
A matrix $W = [w_1, \ldots, w_\ell] \in \R^{n \times \ell}$ with $n = \hat{n} \tilde{n}$ is stored in block low-rank format as a triplet $\mathbf{W} = (U, \Sigma, V)$, where $U \in \mathbb{R}^{\hat{n} \times \hat{r}}$, $V \in \mathbb{R}^{\tilde{n} \times \tilde{r}}$, and the order-3 tensor $\Sigma\in \mathbb{R}^{\hat{r} \times \tilde{r}\times \ell} $ are such that 
\begin{equation} \label{eq:decompwj}
    w_j=\Vec{\big(U\Sigma(j)V^T\big)}, \quad \text{for } j=1,\ldots \ell.
\end{equation} 
Here $\Sigma(j)\in \mathbb{R}^{\hat{r}\times \tilde{r}}$ is a slice of $\Sigma$: $\Sigma(j)_{i_1,i_2}=\Sigma_{i_1,i_2,j}$. 
An equivalent way of viewing~\eqref{eq:decompwj} is to reshape each $w_j$ as an $\hat{n} \times \tilde{n}$ matrix $W_j$, such that $w_j = \Vec(W_j)$, which yields $W_j = U\Sigma(j)V^T$. Yet another way to view~\eqref{eq:decompwj} is to reshape $W$ as an $\hat{n} \times \tilde{n} \times \ell$ tensor and consider~\eqref{eq:decompwj} as a Tucker decomposition of multilinear rank at most $(\tilde{r}, \hat r, \ell)$; see~\cite{Kolda2009}.
In the following, the numbers $\tilde{r}$ and $\hat{r}$ are referred to as ranks; they determine the storage efficiency of the format.
\end{paragraph}

\begin{paragraph}{Application of operators.}
 The format described above is particularly suitable when working with matrices that have Kronecker product structure. Assume that $A = \sum_{i=1}^s \tilde{A}_i \kron \hat{A}_i$, where $\tilde{A}_i \in \R^{\tilde{n} \times \tilde{n}}$, $\tilde{A}_i \in \R^{\hat{n} \times \hat{n}}$. To compute the product $Z=AW$ where $W$ and $Z$ are stored in block low-rank format, which we indicate using $\mathbf{Z} = A \mathbf{W}$, note that
$$
    Aw_j
        = \sum^s_{i=1} (\tilde{A}_i \otimes \hat{A}_i)w_j
        = \sum^s_{i=1} (\tilde{A}_i \otimes \hat{A}_i)\Vec{(U\Sigma(j)V^T)}
        = \sum^s_{i=1} \Vec{\big((\hat{A}_i{U})\Sigma(j)(\tilde{A}_i V)^T\big)}.
$$ 
The product $\mathbf{Z} = (\bar{U}, \bar{\Sigma}, \bar{V})$ is therefore computed as
$$
    \bar{U} = [\hat{A}_1 U, \;\ldots, \; \hat{A}_s U ], 
    \quad 
    \bar{V} = [\tilde{A}_1 V, \;\ldots, \; \tilde{A}_s V ],
    \quad 
    \bar{\Sigma}(j) = \diag(\Sigma(j), \Sigma(j), \ldots, \Sigma(j)). 
$$
When $\max\{\hat{r},\tilde{r}\}$ and $s$ are much smaller than $\max\{\hat{n},\tilde{n}\}$, the complexity of computing $AW$ is significantly reduced from $\mathcal{O}(\ell \hat{n}^2\tilde{n}^2)$ to $\mathcal{O}\big(\ell s(\hat{n}^2\hat{r}+\tilde{n}^2\Tilde{r}+\tilde{n}\hat{n}\min\{\hat{r},\tilde{r}\})\big)$.
\end{paragraph}

\begin{paragraph}{Addition.}
 The addition of two matrices $\mathbf{W}_1 = (U_1, \Sigma_1, V_1)$ and $\mathbf{W}_2 = (U_2, \Sigma_2, V_2)$, denoted by $\mathbf{W} = \mathbf{W}_1 + \mathbf{W}_2$, can be done similarly: The $j$-th column of $W_1+W_2$ is
$$
    \Vec\left(
        \mb{cc} U_1 & U_2 \me 
        \mb{cc} \Sigma_1(j) & 0 \\ 0 & \Sigma_2(j) \me 
        \mb{cc} V_1 & V_2 \me^T
    \right),
$$
so we can set $\mathbf{W} = (U, \Sigma, V)$ with $U = [U_1, \, U_2]$, $V = [V_1, \, V_2]$, $\Sigma(j) = \textrm{diag}(\Sigma_1(j), \Sigma_2(j))$. 
\end{paragraph}

\begin{paragraph}{Recompression.}
 The two operations described above have the undesired effect that the ranks of the result increase compared to the ranks of the inputs; leading to reduced storage and computational efficiency of the format. To control the growth of the ranks, one can perform truncation. For that purpose, we exploit the relation to order-3 tensor discussed above, and perform a higher-order singular value decomposition~\cite{De2000multilinear,Tucker1966Some} where the truncation is limited to the first two modes. We summarize the procedure in Algorithm~\ref{alg:block trucation}. Given a truncation tolerance $\epsilon$ and maximal rank $r_{max}$, Algorithm \ref{alg:block trucation} returns the truncated block low-rank matrix $\mathcal{T}(\mathbf{X})$. The algorithm attempts to ensure $\|\mathcal{T}(\mathbf{X})-\mathbf{X}\|_F \leq \epsilon \|\mathbf{X}\|_F$ while imposing a hard upper limit $\max\{\tilde{r},\hat{r}\}\leq r_{\max}$ for the ranks of the truncated matrix. We used the $n$-mode product of a tensor and a matrix in the last line of Algorithm \ref{alg:block trucation} to calculate the tensor $\Sigma_T$. The definition of $n$-mode product between the order-$N$ tensor $S\in \mathbb{R}^{I_1\times I_2 \times\cdots \times I_n\times\cdots \times I_N}$ and the matrix $M\in \mathbb{R}^{ J_n \times I_n}$ is defined as $S \times_n M \in \mathbb{R}^{I_1\times I_2 \times\cdots \times J_n\times\cdots \times I_N}$ where
\begin{equation}
    \label{eq:definition of n mode}
     (S \times_n M)_{i_1,\cdots,i_{n-1},j_n, i_{n+1}, \cdots, i_{N}}=\sum_{i_n=1}^{I_n}S_{i_1,\cdots ,i_{n-1},i_n ,i_{n+1},\cdots, i_{N}}M_{j_n,i_n}.
\end{equation}
\end{paragraph}

\begin{algorithm}[H]
\caption{Truncation $\mathcal{T}$ of a block low-rank matrix}\label{alg:block trucation}
\algorithmicrequire $X\in \mathbb{R}^{n \times \ell}$ in block low-rank matrix format $\mathbf{X}=(U, \Sigma, V)$; truncation tolerance $\epsilon$; maximal rank $r_{\max}$.\\
\algorithmicensure  $\mathcal{T}(\mathbf{X}) = (U_T, \Sigma_T, V_T)$ in block low-rank matrix format with $U_T \in \mathbb{R}^{\hat{n} \times \hat{r}}$, $V_T \in \mathbb{R}^{\tilde{n} \times \tilde{r}}$ and $\Sigma_T\in \mathbb{R}^{\hat{r} \times \tilde{r}\times \ell}$ such that $\max\{\tilde{r},\hat{r}\}\leq r_{max}$ and 
$\mathcal{T}(\mathbf{X}) \approx \mathbf{X}$.

     \begin{algorithmic}[1]
     \State Compute economy-size QR factorization $U=Q_uR_u$.
     \State Compute economy-size QR factorization $V=Q_vR_v$.
     \State Compute SVD $[R_u\Sigma(1)R_v^T,\cdots, R_u\Sigma(\ell
     )R_v^T]=U_1S V_1^T$.
     \State Find integer ${r}$ such that $\sqrt{(\|S\|_F^2-\sum_{i=1}^{{r}} S^2_{ii})} \leq \frac{\epsilon}{\sqrt{2}} \|S\|_F$ and set $\tilde{r}=\min\{r,r_{max}\}$.
     \State Set $U_T=Q_uU_1(:,1:\Tilde{r})$.
     \State Compute SVD $[R_v\Sigma(1)^TR_u^T,\cdots, R_v\Sigma(\ell)^TR_u^T]=U_2S V_2^T$.
     \State Find integer ${r}$ such that $\sqrt{(\|S\|_F^2-\sum_{i=1}^{{r}} S^2_{ii})} \leq \frac{\epsilon}{\sqrt{2}} \|S\|_F$ and set $\hat{r}=\min\{r,r_{max}\}$.
     \State Set $V_T=Q_vU_2(:,1:\hat{r})$.
     \State Compute $\Sigma_T=\Sigma\times_1 U_1(:,1:\Tilde{r})^T \times_2 U_2(:,1:\hat{r})^T$\Comment{See \eqref{eq:definition of n mode} for definition.} 

     \end{algorithmic}
\end{algorithm}

\begin{paragraph}{Block inner product.}
 We next consider the computation of the matrix $Z=W_1^T W_2$ where $W_1$, $W_2$ are stored in block low-rank format, and the result is stored in an array, i.e., we compute $Z=\mathbf{W_1}^T \mathbf{W_2}$. 
If $\mathbf{W}_1 = (U_1, \Sigma_1, V_1)$ and $\mathbf{W}_2 = (U_2, \Sigma_2, V_2)$, we observe that
    \begin{align*}       
        Z_{i_1, i_2} 
            &= \big(W_1^T W_2\big)_{i_1,i_2} \\
            &= \left( \Vec{\big(U_1\Sigma_1(i_1)V_1^T\big)} \right)^T\Vec{\big(U_2\Sigma_2(i_2)V_2^T\big)} \\
            &= \big( (V_1 \kron U_1) \Vec{\big(\Sigma_1(i_1)\big)} \big)^T \big( (V_2 \kron U_2) \Vec{\big(\Sigma_2(i_2)\big)}  ) \\
            &= \Vec{\big(\Sigma_1(i_1)\big)}^T (V_1^T V_2 \kron U_1^T U_2) \Vec{\big(\Sigma_2(i_2)\big)} \\
            &= \Vec{\big(\Sigma_1(i_1)\big)}^T \Vec{\big(U_1^T U_2 \Sigma_2(i_2) V_2^T V_1\big)} \\
            &= \text{trace} \big( \Sigma_1(i_1)^T (U_1^T U_2) \Sigma_2(i_2) (V_2^T V_1) \big).
    \end{align*}
This operation has complexity of $\mathcal{O}(r^3 \ell^2 + \max\{\tilde{n}, \hat{n}\}r^2)$, where $r$ is the maximal rank of $\mathbf{W}_1$, $\mathbf{W}_2$.

In the special case when ${\bf W_1}={\bf W_2}={\bf W}=(U, \Sigma, V)$ and the matrices $U$ and $V$ have orthonormal columns, the formula above simplifies to
\begin{equation}
    \label{eq: gram matrix}
    Z_{i_1, i_2}
        = \big(W^T W\big)_{i_1,i_2} = \text{trace}\big(\Sigma(i_1)^T\Sigma(i_2)\big),
\end{equation} 
and the complexity is reduced to $\mathcal{O}(r^3 \ell^2)$.
\end{paragraph}

\begin{paragraph}{Matrix multiplication.}
 Finally, we explain how to do matrix multiplication $Z=WB$ when $W$ and $Z$ are stored in block low-rank format, and $B\in \mathbb{R}^{\ell \times \ell}$. In other words, we compute $\mathbf{Z} = \mathbf{W}B = (\bar{U}, \bar{\Sigma}, \bar{V})$. The factors of $\mathbf{Z}$ can be set as
\begin{equation}
\label{eq: blk matrix mutiplication}
    \bar{U}=U,\quad \bar{V}=V,\quad \text{and}\quad \bar{\Sigma}(i)=\sum^\ell_{j=1}B_{ji}\Sigma(j).
\end{equation}
\end{paragraph}

\begin{paragraph}{Initial matrix.}
 With the described format, using the Khatri--Rao product of two Gaussian matrices for the initial matrix $X^{(0)}$ in Algorithm \ref{alg:lobpcg} is a natural choice. Assume $X^{(0)}=\hat{\Omega} \odot \tilde{\Omega}\in \mathbb{R}^{n\times \ell}$: we compute the QR-factorizations $\hat{\Omega} = \hat{Q} \hat{R}$, $\tilde{\Omega} = \tilde{Q} \tilde{R}$ and set $U = \hat{Q} \in \R^{\hat{n} \times \ell}$, $V = \tilde{Q} \in \R^{\tilde{n} \times \ell}$, and $\Sigma(j) = \hat{R}e_j (\tilde{R}e_j)^T \in \R^{\ell \times \ell}$. The triplet $\mathbf{X}^{(0)} = (U, \Sigma, V)$ is now a low-storage representation of the initial matrix which can be used as input to Algorithm \eqref{alg:lobpcg}.

While the theory from Section~\ref{sec:contour} indicates that this choice of initial matrix $X^{(0)}$ for Algorithm~\ref{alg:lobpcg} is not unreasonable, it is difficult to turn this intuition into a precise mathematical statement; mainly due to the lack of a complete convergence theory for LOBPCG. Even for the simpler case of the preconditioned inverse subspace iteration (PINVIT), the existing convergence analyses~\cite{Knyazev2003geometric, Zhou2019Cluster} impose strong conditions on the initial matrix to guarantee convergence to the smallest eigenvalue(s). These conditions are rarely met for any of the commonly chosen initial matrices, but they also do not seem to be needed for the global convergence of PINVIT. 
\end{paragraph}

\begin{paragraph}{Orthonormalization.}
 In order to prevent the occurrence of an ill-conditioned Gram matrix in \eqref{eq: generalized eigenvalues problem}, as suggested in \cite{Knyazev2007Block}, we ensure that the column vectors inside each $P^{(i)}$ and $R^{(i)}$ are orthonormalized by using the Cholesky decomposition of the Gram matrix $(P^{(i)})^T P^{(i)}$ and $(R^{(i)})^T R^{(i)}$ respectively. The matrix $X^{(i)}$ remains close to being orthonormal at the end of each iteration so we do not subject it to this procedure. The described technique may perform poorly in some cases compared to the Hetmaniuk-Lehoucq orthogonalization strategy \cite{Hetmaniuk2006Basis}. However, it is an inexpensive and easy to implement approach when working with the block low-rank matrix format.
 \end{paragraph}

\begin{remark}
    In Section \ref{sec:contour}, at the end of Algorithm \ref{alg:feast} we needed to compute Rayleigh--Ritz approximation from the subspace $\spann\{x_1, \ldots, x_\ell\}$ where $x_j = \Vec{(\tilde{X}_j \hat{X}_j^\ast)}$. This can be done efficiently by storing $[x_1, \ldots, x_\ell]$ in block low-rank format as $\mathbf{X} = (U, \Sigma, V)$ with $U = [\tilde{X}_1, \ldots, \tilde{X}_\ell]$, $V = [\hat{X}_1, \ldots, \hat{X}_\ell]$, $\Sigma(j) = \diag(0_{r_1}, \ldots, 0_{r_{j-1}}, I_{r_j}, 0_{r_{j+1}}, \ldots, 0_{r_\ell})$, where $r_j$ is the number of columns in $\tilde{X}_j$ and $\hat{X}_j$. After immediate truncation, we can orthonormalize $\mathbf{X}$ and compute the Rayleigh quotient $\mathbf{X}^\ast A \mathbf{X}$ using the operations with block low-rank format as described above. 
    Note that this computation only needs to be done once in Algorithm \ref{alg:feast}, upon completing the quadrature.
\end{remark}

\subsection{Numerical experiments}

In this section we present numerical experiments with Algorithm \ref{alg:lobpcg}, applied to eigenvalue problems with the matrix~\eqref{eq:discretized_eigenvalue_problem} from Section~\ref{sec:implementation}, which we obtained by discretizing the Schr\"{o}dinger equation with finite differences.
We aim at obtaining the $k$ smallest eigenvalues and the associated eigenvectors; this time we will study four different potentials $V$. 

We take 3000 discretization points on each axis, resulting in a symmetric matrix $A$ of size $3000^2 \times 3000^2$. As Algorithm \ref{alg:lobpcg} only works for positive definite matrices, we replace $A$ with $A+\sigma I$ using an appropriate shift $\sigma \in \R$ if necessary. Unless mentioned otherwise, we let $k=4$, the block size is $\ell=6$, and $X^{(0)} = \hat{\Omega} \odot \tilde{\Omega}$, where $\hat{\Omega}$ and $\tilde{\Omega}$ are independent Gaussian random matrices of size $\mathbb{R}^{3000 \times 6}$. The preconditioner is set to $M=I\otimes K+K\otimes I$, and Line \ref{alg:lobpcg:precondotioner} is computed by using the alternating-direction implicit (ADI) Sylvester solver \cite[Algorithm 1]{Benner2009ADI} with 8 ADI iterations. The low-rank truncation tolerance and maximal rank in Algorithm \ref{alg:block trucation} are set to $\epsilon=10^{-7}$ and $r_{max}=50$, respectively. For each iteration we report: 
\begin{itemize}
    \item The residual $\|Ax^{(i)}_j-\lambda^{(i)}_j x^{(i)}_j\|_2$.
    \item The rank $\max\{\Tilde{r},\hat{r}\}$ of the block low-rank matrix format matrix $\mathbf{X}^{(i)}$ after each iteration.
    \item Estimated absolute eigenvalue error $|\lambda_j-\lambda^{(i)}_j|$. As the exact value of $\lambda_j$ is not known, we estimate it using Algorithm $\ref{alg:lobpcg}$ with a lower low-rank truncation tolerance of $\epsilon_{ref}=10^{-10}$, without restricting the maximal rank and by running more (140) iterations.
\end{itemize}

\begin{example}
\label{ex:lobpcg-sum-of-squares}
As in Section~\ref{sec:implementation}, we consider the potential
$V(x,y)=(x^2+y^2-xy)/2$ for $(x,y)\in [-1,1]\times [-1,1]$.
Figure \ref{fig: LOBPCG_sum} displays the resulting residual and absolute errors. We observe that the rank increases quite rapidly during the transient phase of the iteration. On the other hand, once convergence is reached, we can observe that the rank drops. This reflects that the eigenvectors admit a good low-rank approximation. The entire computation takes 27.10 seconds (i.e., 0.45s per iteration on average).
\begin{figure}[H]
\includegraphics[width=0.48\textwidth]{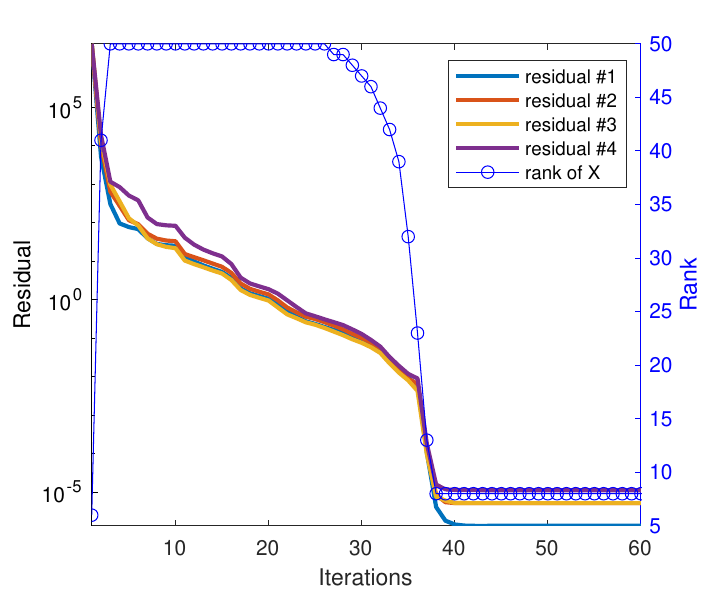}
\hspace{\fill}
\includegraphics[width=0.48\textwidth]{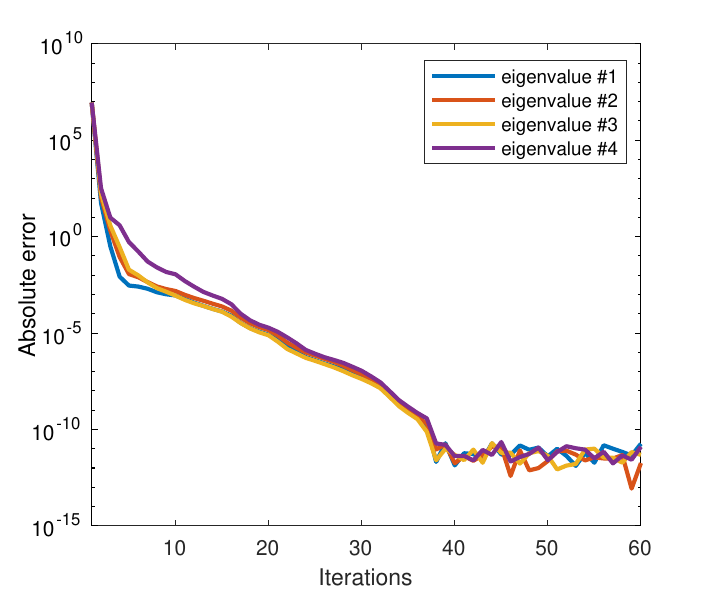}
\caption{Residuals and absolute errors of Algorithm \ref{alg:lobpcg} (LOBPCG with low-rank truncation) applied to Example \ref{ex:lobpcg-sum-of-squares}.}\label{fig: LOBPCG_sum}
\end{figure}
\end{example}

\begin{example}
\label{ex:lobpcg-gaussian}
We now let $V$ be the Gaussian potential 
$
   V(x,y)=-50\exp{(-x^2-y^2)}
$ for $(x,y)\in [-5,5] \times [-5,5]$.
The resulting residuals and absolute errors are reported in Figure \ref{fig: LOBPCG_gauss}. Comparing Figures \ref{fig: LOBPCG_sum} and \ref{fig: LOBPCG_gauss}, the algorithm requires more iterations to reach the asymptotic convergence regime. Although the final approximations have higher rank, the convergence is still satisfactory. The entire computation takes 39.13 seconds (i.e., 0.65s per iteration on average).
\begin{figure}[H]
\includegraphics[width=0.48\textwidth]{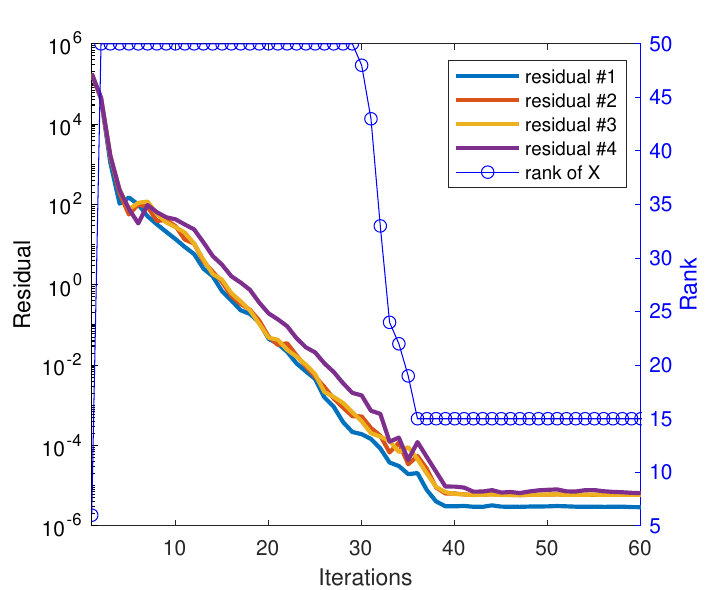}
\hspace{\fill}
\includegraphics[width=0.48\textwidth]{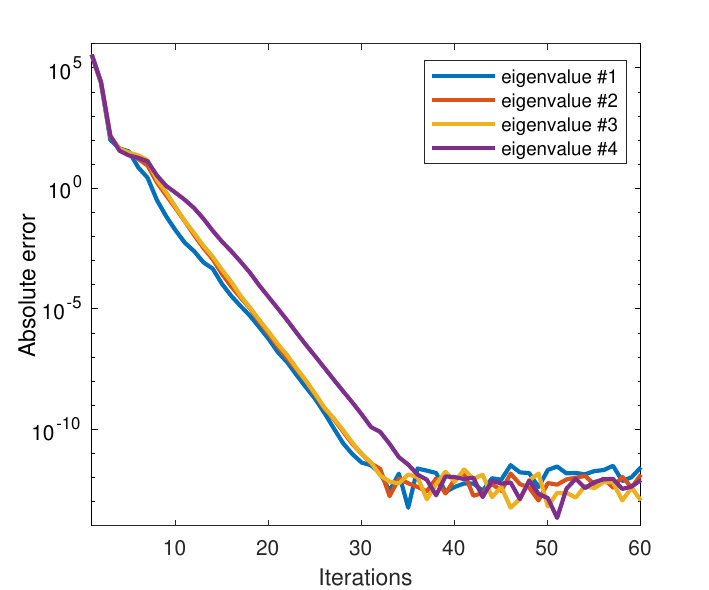}
\caption{Residuals and absolute errors of Algorithm \ref{alg:lobpcg} (LOBPCG with low-rank truncation) applied to Example \ref{ex:lobpcg-gaussian}.}\label{fig: LOBPCG_gauss}
\end{figure}
\end{example}

\begin{example}
\label{ex:lobpcg-mathieu}
For the last example, we let $V$ be a Gaussian perturbation of the Mathieu potential~\cite{boulton2007approximation}:
\begin{equation*}
   V(x,y)=\cos(x)+\cos(y)-6\exp{(-x^2-y^2)},\text{ for } (x,y)\in [-25,25]^2.
\end{equation*} 
Here, our goal is slightly different: we would like to compute the single eigenvalue of $A$ that lies in the middle of the spectrum between two clusters of eigenvalues. To illustrate this, we first discretize the domain using only 100 points on each axis, solve the smaller eigenvalue problem, and plot the 100 smallest eigenvalues of $A$ as purple crosses in Figure \ref{fig: LOBPCG_matieu_eigenvalues}. Our objective is to find the eigenvalue closest to $-0.2$. As Algorithm \ref{alg:lobpcg} cannot target the inner eigenvalues of $A$, we work with the positive definite matrix $\tilde{A}=(A+0.2I)^2$, so that the desired eigenvalue is on the edge of the spectrum. In Figure~\ref{fig: LOBPCG_matieu_eigenvalues}, yellow circles mark the 100 smallest eigenvalues of $\tilde{A}$. The smallest one corresponds to the desired eigenvalue of $A$.
\begin{figure}[H]
\centering
\includegraphics[width=0.75\textwidth]{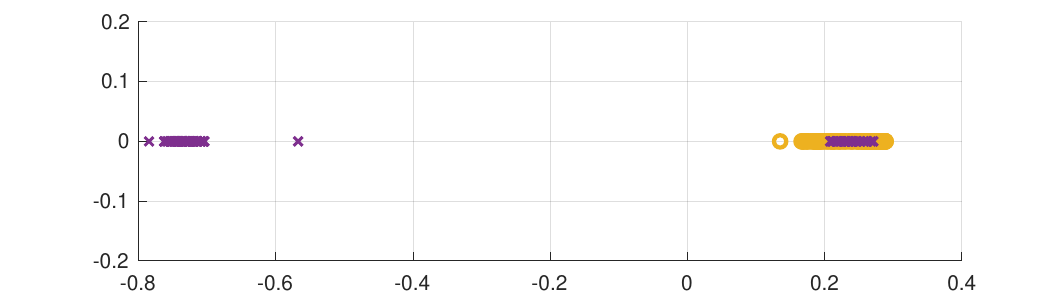}
\caption{ 100 smallest eigenvalues of $A$ and $(A+0.2I)^2$ where $A$ is the eigenvalues problem from Example~\ref{ex:lobpcg-mathieu} with 100 discretization points.}
\label{fig: LOBPCG_matieu_eigenvalues}
\end{figure}
Returning to the original problem where the discretization uses $3000$ points on each axis, resulting in a matrix $A$ of size $3000^2 \times 3000^2$, we apply Algorithm \ref{alg:lobpcg} to $\tilde{A} = (A+0.2I)^2$ with $k=1$ and block size $\ell=3$. 
For this purpose, we first express $A+0.2I$ in the form \eqref{eq:discretized_eigenvalue_problem_kron_form}. The preconditioner used in Algorithm \ref{alg:lobpcg} is then set to $M=I\otimes K^2+ K^2\otimes I$. 
In this example, to obtain an accurate approximation we needed to take a higher maximal rank $r_{max}=120$ and 12 ADI iterations in the Sylvester solver (Line \ref{alg:lobpcg:precondotioner}).
The reference eigenvalue for evaluating the absolute error is obtained by running $600$ iterations of the same algorithm  with the lower low-rank truncation tolerance $\epsilon_{ref}=10^{-10}$, without restricting the maximal rank.

In the left plot of Figure \ref{fig: LOBPCG_matieu}, we display residual norms (with respect to $\tilde A$) for all $\ell=3$ Ritz pairs in each iteration of the algorithm. The absolute eigenvalue errors $|\lam_j - \lam_j^{(i)}|$ with respect to the 3 eigenvalues $\lam_1, \lam_2, \lam_3$ of $A$ closest to $-0.2$ are shown in the right plot. 

This example turns out to be less favorable for Algorithm \ref{alg:lobpcg} compared to the other two examples above. More iterations are required, likely due to the use of a lower-quality preconditioner. However, one can still obtain satisfactory residuals and accurate eigenvalue approximations. The entire computation takes 915.43 seconds (i.e., 2.03s per iteration on average).
\begin{figure}[h]
\includegraphics[width=0.48\textwidth]{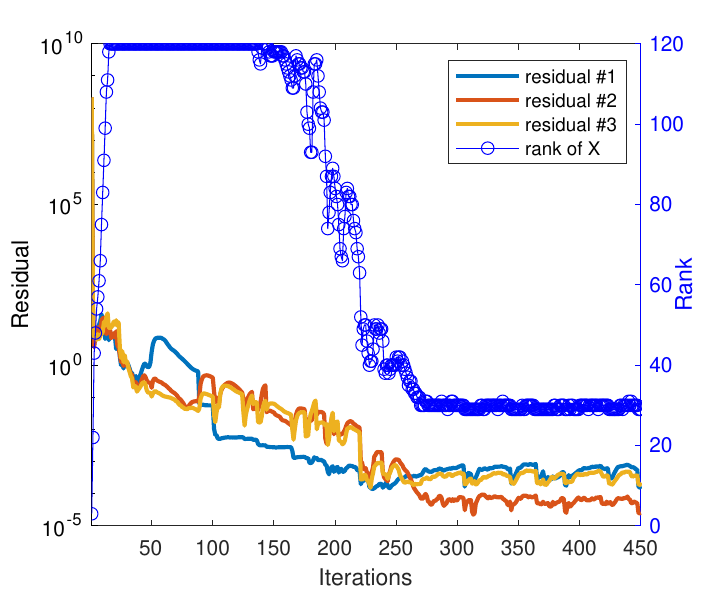}
\hspace{\fill}
\includegraphics[width=0.48\textwidth]{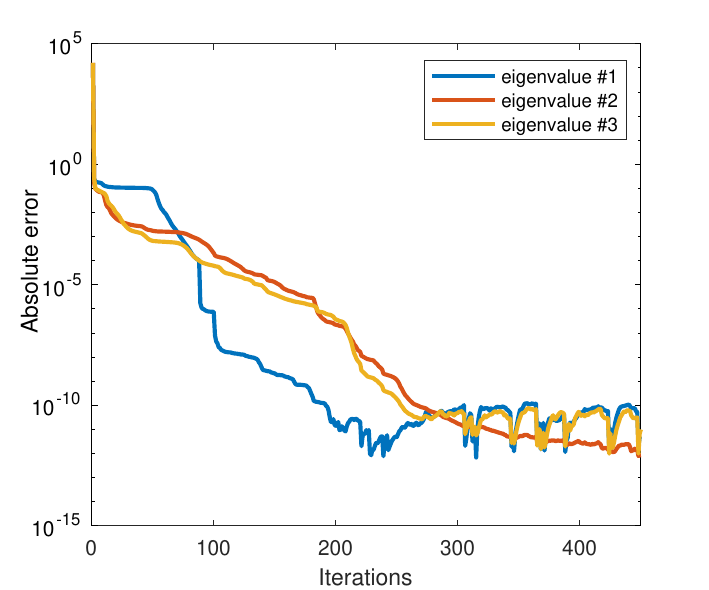}
\caption{Residual with respect to $\Tilde{A}$ and absolute errors of the computed eigenvalues when Algorithm \ref{alg:lobpcg} (LOBPCG with low-rank truncation) is applied to Example~\ref{ex:lobpcg-mathieu}.\label{fig: LOBPCG_matieu}}
\end{figure}

\end{example}

\section{Conclusions}
In this paper, we have analyzed the use of random Khatri–Rao products as embeddings and starting blocks in iterative eigenvalue solvers for Hermitian matrices defined as short sums of Kronecker products. Our technique can be formally extended, but with weaker bounds involving some type of condition number measuring the nonnormality of the matrix, to the problem of approximating a collection of semisimple eigenvalues of a non-Hermitian short sum of Kronecker products.
However, the viability of such a method strongly depends on the availability of an efficient Sylvester equation solver.

\begin{paragraph}{Acknowledgments.}
    The authors thank Patrick K\"{u}rschner for providing the Matlab implementation of the Sylvester ADI method \cite{Kuerschner16}, which is used as the preconditioner in Sections \ref{sec:contour} and \ref{sec:lobpcg}.
    DK thanks Felix Krahmer for helpful discussions on Khatri-Rao structured embeddings.
\end{paragraph}

\section{Appendix}
\label{sec:appendix}

The purpose of this section is to prove Theorem~\ref{Thm: JLmoment}. The structure of the proof follows the one of Theorem 42 from \cite{ahle2020oblivious}; the main difference is that we make all steps explicit in order to keep track of the constants. To establish the proof, we need some preliminary results. For a random variable $Z$, we let $\|Z\|_{L^p} := (\mathbb{E}|Z|^p)^{1/p}$ for $p\in\R$ with $p\geq 1$.
\begin{theorem}[{\cite[Theorem 2]{latala1997estimation}}]
\label{Thm: latala}
    Let $X_1,\ldots, X_\ell$ be a sequence of independent symmetric random variables, and $p\geq 2$. Then the following inequalities hold:
    \begin{equation*}
        \frac{e-1}{2e^{2}}\vvvert(X_i)\vvvert_p\leq\|X_1+\cdots+X_\ell \|_{L^p}\leq e \vvvert(X_i)\vvvert_p,
    \end{equation*} where $\vvvert(X_i)\vvvert_p:=\inf\Big\{t>0\colon  \sum_{i}\ln\mathbb{E}\Big(\Big|1+\frac{X_i}{t}\Big|^p\Big)\leq p\Big\}$.
\end{theorem}
The proof of the following lemma closely follows the proof of \cite[Corollary 1]{latala1997estimation}.
\begin{lemma}
\label{lemma: orlicz_upper_bound}
    Let $X, X_1,\ldots, X_\ell$ be a sequence of i.i.d. symmetric random variables. Then, for any integer $p\geq 2$,
    \begin{equation*}
        \vvvert(X_i)\vvvert_p\leq 2e\cdot \sup \Big\{\frac{p}{s}\Big(\frac{\ell}{p}\Big)^\frac{1}{s}\|X\|_{L^s}\colon \max\Big(2,\frac{p}{\ell}\Big)\leq s\leq p\Big\}.
    \end{equation*}
\end{lemma}
\begin{proof}
    We define the following functions on $\mathbb{R}$ for $p>0$:
    \begin{equation} \label{eq:defphi}
        \varphi_p(x)=|1+x|^p \text{ and } \tilde{\varphi}_p(x)=\frac{\varphi_p(x)+\varphi_p(-x)}{2} = \begin{cases}
  \frac{(1+|x|)^p+(1-|x|)^p}{2}& |x|\leq1 \\
 \frac{(1+|x|)^p+(|x|-1)^p}{2} & |x|>1
\end{cases}.
    \end{equation}
Using that the variables $X_i$ are i.i.d. and symmetric, it follows that
\begin{eqnarray}
    \vvvert(X_i)\vvvert_p & = & \inf\Big\{t>0: \sum_{i}\ln\mathbb{E}\Big(\Big|1+\frac{X_i}{t}\Big|^p\Big)\leq p\Big\}=\inf\big\{t>0: \mathbb{E}\big(\varphi_p(X/t)\big)\leq e^{p/\ell}\big\} \nonumber \\
    &=& \inf\big\{t>0\colon  \mathbb{E}\big(\tilde \varphi_p(X/t)\big)\leq e^{p/\ell}\big\} \label{eq:exprXphi}
\end{eqnarray} Setting
\begin{equation*}
     \hat t = \sup \Big\{\frac{p}{s}\Big(\frac{\ell}{p}\Big)^\frac{1}{s}\|X\|_{L^s}\colon \max\Big(2,\frac{p}{\ell}\Big)\leq s\leq p\Big\},
\end{equation*}
we can conclude the proof by showing that $2 e \hat t$ is in the admissible set of~\eqref{eq:exprXphi}, that is,
$\mathbb{E}\big(\tilde \varphi_p(X/(2 e \hat t))\big)\leq e^{p/\ell}$. For this purpose, we will make use of the inequality
\begin{equation*}
    \tilde{\varphi}_p(x)\leq 1+\sum_{2\leq k<p}{p\choose k}|x|^k+|x|^p,
\end{equation*}
which follows from the second expression for $\tilde \varphi$ in~\eqref{eq:defphi}.

Assuming $p/\ell \le 2$ and using ${p\choose k}\leq (\frac{ep}{k})^k$, it then follows that
\begin{align*}
    \mathbb{E}\tilde{\varphi}_p\Big(\frac{X}{2e\hat{t}}\Big)&\leq 1+\sum_{2\leq k<p} \frac{p^k}{(2\hat{t}k)^k}\|X\|^k_{L^k}+\frac{\|X\|^p_{L^p}}{(2e\hat{t})^p} \leq 1+\sum_{2\leq k<p} \frac{1}{2^k}\frac{p}{\ell}+\frac{1}{(2e)^p}\frac{p}{\ell}\\ &\leq 1+\frac{p}{\ell}\sum_{2\leq k\leq p} \frac{1}{2^k} \leq 1+\frac{p}{\ell} \leq e^{p/\ell}.
\end{align*} 
For the case $p/\ell\geq 2$, we use $(\frac{p}{\ell})^{\frac{\ell}{p}}\leq e$ and, again, ${p\choose k}\leq (\frac{ep}{k})^k$ to obtain 
\begin{align*}
       \mathbb{E}\tilde{\varphi}_p\Big(\frac{X}{2e\hat{t}}\Big)&\leq 1+ \sum_{2\leq k\leq p/\ell}  {p\choose k}\frac{\|X\|^k_{L^k}}{(2e\hat{t})^k}+\sum_{p/\ell<k<p} \frac{p^k}{(2\hat{t}k)^k}\|X\|^k_{L^k}+\frac{\|X\|^p_{L^p}}{(2e\hat{t})^p}\\&\leq 1+\sum_{1\leq k\leq p/\ell} \frac{p^k}{k!}\frac{\|X\|^k_{L^k}}{(2e\hat{t})^k}+\sum_{p/\ell<k\leq p} \frac{1}{2^k}\frac{p}{\ell}
       \\&\leq  \exp(p\|X/(2e\hat{t})\|_{L^{p/\ell}}+\sum_{p/\ell<k\leq p}  \frac{1}{2^k}\frac{p}{\ell} \leq e^{p/(2\ell)}+\frac{p}{\ell} \leq e^{p/\ell}.      
\end{align*}
This completes the proof.
\end{proof}
The following corollaries provide upper bounds on the $L^p$ norm for sums of i.i.d. mean zero random variables.
\begin{corollary}
\label{corollary: sum_Lp}
Let $Z, Z_1,\ldots, Z_\ell$ be a sequence of i.i.d mean-zero random variables, $\epsilon_1,\ldots \epsilon_\ell$ be Bernoulli sequence independent of $Z$ and integer $p\geq 2$. Suppose that $\|Z\|_{L^p}<\infty$, then
\begin{equation*}
    \|Z_1+\cdots+Z_\ell\|_{L^p}\leq 2\|\epsilon_1 Z_1+\cdots+\epsilon_\ell Z_\ell \|_{L^p}\leq 4e^2 \sup \Big\{\frac{p}{s}\Big(\frac{\ell}{p}\Big)^\frac{1}{s}\|Z\|_{L^s}:\max\Big(2,\frac{p}{\ell}\Big)\leq s\leq p\Big\}.
\end{equation*}
\end{corollary}
\begin{proof}
    The first inequality follows from symmetrization \cite[Lemma 6.3]{LedouxProbability2011} and the second inequality follows from Theorem \ref{Thm: latala} and Lemma \ref{lemma: orlicz_upper_bound}.
\end{proof}
\begin{corollary}
\label{corollary: sum of rv}
    Let $p\geq2$ be an integer and let $Z_1,\ldots, Z_\ell$ be i.i.d. mean zero random variables such that $\|Z\|_{L^s}\leq {(as)}^{2}$ for all $s\geq 1$ and some $a>0$. Then
    \begin{equation*}
        \|Z_1 + \cdots + Z_\ell \|_{L^p}\leq 4e^2  \max\Big\{\frac{1}{2}(2a)^2\sqrt{p\ell},\Big(\frac{\ell}{p}\Big)^{\frac{1}{p}}(ap)^2\Big\}.
    \end{equation*}
\end{corollary}
\begin{proof}
     Considering the function $h(s)=(\frac{\ell}{p})^\frac{1}{s}s$, we see that its second derivative
     $h^{\prime\prime}(s)=(\frac{\ell}{p})^\frac{1}{s}\frac{1}{s^3} \log^2(\frac{\ell}{p})$ is positive and, hence, the set
     $\big\{\frac{p}{s} (\frac{\ell}{p})^\frac{1}{s}(as)^2\colon \max(2,\frac{p}{\ell})\leq s\leq p\big\}$ attains its maximum on the boundary of $s$. Inserting this observation into the bound of  Corollary \ref{corollary: sum_Lp} completes the proof.
\end{proof}
Then following lemma states a basic inequality on inner products with standard Gaussian random vectors.
\begin{lemma}
\label{Lemma: Gassuian lp}
Consider independent $Z_1,\ldots, Z_\ell \sim N(0,1)$ and let $a=(a_1,\ldots, a_\ell)\in \mathbb{R}^{\ell}$.  Then 
$
    \|a_1 Z_1 + \cdots + a_\ell Z_\ell \|_{L^p}\leq \sqrt{p}\|a\|_2
$
holds for every $p\geq 1$.
\end{lemma}
\begin{proof}
    Because $Y:= a_1 Z_1 + \cdots + a_\ell Z_\ell \sim N(0,\|a\|^2_2)$ and $\|Y \|_{L^p} = \sqrt{2} \big( \Gamma(\frac{p+1}{2}) \big)^{1/p} (\sqrt{\pi} )^{-1/p} \|a\|_2$, the result is shown by bounding the Gamma function $\Gamma(s)$ for $s = (p+1)/2$. For $1\le s \le 2$, we have $0\le \Gamma(s) \le 1$ and, hence, $\|Y \|_{L^p} \le \sqrt{2} (\sqrt{\pi} )^{-1/p} \|a\|_2 \le \sqrt{p} \|a\|_2$ holds for $1\le p \le 3$. For $p \ge 3$, we use $\Gamma(s) = (s-1) \Gamma(s-1)$, decompose $s = k + \tilde s$ with $k = \lfloor s \rfloor - 1 \in \mathbb N$, $\tilde  s \in [1,2]$, and obtain
    \[
     \Gamma(s) = (s-1)(s-2)\cdots (s-k) \Gamma(\tilde s) \leq (p / 2)^{k}\leq (p/2)^{p/2},
    \]
    which implies $\|Y\|_{L^p} \le \sqrt{p} (\sqrt{\pi} )^{-1/p} \|a\|_2 \le \sqrt{p} \|a\|_2$ concludes the proof.
\end{proof}

The following lemma is a special case of \cite[Lemma 19]{ahle2020oblivious}.

\begin{lemma}
\label{lemma:knocker_Ls}
    Let $\tilde{\omega}\in \mathbb{R}^{\tilde{n}}$, $\hat{\omega}\in \mathbb{R}^{\hat{n}}$ be independent standard Gaussian random vectors and let $a\in \mathbb{R}^{\hat{n}\Tilde{n}}$. Then 
    $
        \|\langle \tilde{\omega}\otimes  \hat{\omega},a\rangle\|_{L^s}\leq s\|a\|_2
    $
    holds for every $s\geq 1$.
\end{lemma}
\begin{proof}
Letting $A \in \R^{\hat n\times \tilde n}$ such that $a = \Vec(A)$, we have $\langle \tilde{\omega}\otimes  \hat{\omega},a\rangle = \hat{\omega}^T A \tilde{\omega}$. Using the independence of $\tilde\omega$ and $\hat \omega$, we obtain the result by applying Lemma~\ref{Lemma: Gassuian lp} twice:
\begin{eqnarray*}
  \| \hat{\omega}^T A \tilde{\omega} \|_{L^s} &=& \|\langle A^T \hat{\omega}, \tilde{\omega} \rangle \|_{L^s} \le \sqrt{p} \| \| A^T \hat{\omega} \|_2 \|_{L^s}
  = \sqrt{p} \big\| \big( \langle a_1, \hat{\omega} \rangle^2  + \cdots +  \langle a_{\tilde n}, \hat{\omega} \rangle^2 \big)^{1/2} \big\|_{L^s} \\
  &\le & \sqrt{p} \big( \big\|  \langle a_1, \hat{\omega} \rangle^2  + \cdots +  \langle a_{\tilde n}, \hat{\omega} \rangle^2 \big\|_{L^s/2} \big)^{1/2}
  \le \sqrt{p} \big( \big\|  \langle a_1, \hat{\omega} \rangle^2 \big\|_{L^s/2}  + \cdots +  \big\| \langle a_{\tilde n}, \hat{\omega} \rangle^2 \big\|_{L^s/2} \big)^{1/2} \\
  &= &  \sqrt{p} \big( \big\|  \langle a_1, \hat{\omega} \rangle \big\|^2_{L^s}  + \cdots +  \big\| \langle a_{\tilde n}, \hat{\omega} \rangle \big\|^2_{L^s} \big)^{1/2} \le 
  p \|A\|_F = p\|a\|_2,
\end{eqnarray*}
where $a_i$ denotes the $i$th column of $A$.
\end{proof}
Now, we have all ingredients for the proof of Theorem \ref{Thm: JLmoment}.
\begin{proof}[Proof of Theorem \ref{Thm: JLmoment}]
 Assume $\|x\|_2=1$, denote the $i$th columns of $\tilde{\Omega}$ and $\hat{\Omega}$ as $\tilde{\omega}_i$ and $\hat{\omega}_i$, respectively, which are independent standard Gaussian vectors. Then \begin{equation*}
     \mathbb{E}[\|\Omega^T x\|^2_2]=\frac{1}{\ell}\sum^\ell_{i=1}\mathbb{E}[ \langle \tilde{\omega}_i\otimes \hat{\omega}_i, x \rangle^2]=\frac{1}{\ell}\sum^\ell_{i=1}\|x\|_2^2=1.
 \end{equation*}
In order to get bounds on the higher moments, 
 let $Z_i=\langle \tilde{\omega}_i\otimes \hat{\omega}_i, x \rangle^2-1$.  For all $s\geq 1$,
\begin{equation*}
\label{eq:zineq}
    \|Z_i\|_{L^s}=\|\langle \tilde{\omega}_i\otimes \hat{\omega}_i, x \rangle^2-1\|_{L^s}\leq 2 {\|\langle \tilde{\omega}_i\otimes \hat{\omega}_i, x \rangle^2}\|_{L^s}=2\|\langle \tilde{\omega}_i\otimes \hat{\omega}_i, x \rangle\|^2_{L^{2s}} \le 8 s^2,
\end{equation*}
where we used Lemma \ref{lemma:knocker_Ls}, the triangle inequality, and $1 = \mathbb{E}{\langle \tilde{\omega}_i\otimes \hat{\omega}_i, x \rangle^2} \le \|\langle \tilde{\omega}_i\otimes \hat{\omega}_i, x \rangle^2 \|_{L^s}$.
This allows us to apply Corollary \ref{corollary: sum of rv} with $a = 2 \sqrt{2}$ to obtain
\begin{align*}
   \big\|\frac{1}{{\ell}}\|(\tilde{\Omega} \odot \hat{\Omega})^Tx\|_2^2-1\big\|_{L^p}
   &=\big\|\frac{1}{{\ell}}(Z_1 + \cdots + Z_{\ell}) \big\|_{L^p}
   \leq 4e^2 \max\Big\{16\sqrt{p/\ell},8\Big(\frac{\ell}{p} \Big)^{\frac{1}{p}} p^2 / \ell \Big\}\\
   &\leq 4e^2 \max\big\{32\sqrt{p/\ell},8{\ell}^{\frac{1}{p}}p^2 / \ell \big\} \leq 128e^2 \max\{\sqrt{p/\ell},p^2 / \ell \},
\end{align*} where the last inequality follows from the fact that if $32\sqrt{p/\ell}\leq 8{\ell}^{\frac{1}{p}}{p^2}/{\ell}$
then $
    \ell^{\frac{1}{p}}\leq (\frac{p}{2})^{3/(p-2)}\leq 4$ for all $p\geq 4.
$ Note that $p= \lceil\frac{1}{2}\log(\frac{1}{\delta})\rceil\geq 4$.  Choosing $\ell\geq \max\{(128e^4)^2p\varepsilon^{-2},(128e^4)p^2\varepsilon^{-1}\}$, we obtain 
\[
    \big\|\frac{1}{{\ell}}\|(\tilde{\Omega} \odot \hat{\Omega})^Tx\|_2^2-1\big\|_{L^p} \leq \varepsilon e^{-2} \leq\varepsilon \delta^{1/p},
\] which is exactly the JL-moment property.
\end{proof}

\bibliographystyle{plain}
\bibliography{references}

\begin{thebibliography}{10}

\bibitem{ahle2020oblivious}
Thomas~D. Ahle, Michael Kapralov, Jakob B.~T. Knudsen, Rasmus Pagh, Ameya
  Velingker, David~P. Woodruff, and Amir Zandieh.
\newblock Oblivious sketching of high-degree polynomial kernels.
\newblock In {\em Proceedings of the 2020 {ACM}-{SIAM} {S}ymposium on
  {D}iscrete {A}lgorithms}, pages 141--160. SIAM, Philadelphia, PA, 2020.

\bibitem{Sakurai-Block}
Junko Asakura, Tetsuya Sakurai, Hiroto Tadano, Tsutomu Ikegami, and Kinji
  Kimura.
\newblock A numerical method for nonlinear eigenvalue problems using contour
  integrals.
\newblock {\em JSIAM Lett.}, 1:52--55, 2009.

\bibitem{Bamberger22}
Stefan Bamberger, Felix Krahmer, and Rachel Ward.
\newblock Johnson-{L}indenstrauss embeddings with {K}ronecker structure.
\newblock {\em SIAM J. Matrix Anal. Appl.}, 43(4):1806--1850, 2022.

\bibitem{BBK18}
Casey Battaglino, Grey Ballard, and Tamara~G. Kolda.
\newblock A practical randomized {CP} tensor decomposition.
\newblock {\em SIAM J. Matrix Anal. Appl.}, 39(2):876--901, 2018.

\bibitem{BennerBreiten13}
Peter Benner and Tobias Breiten.
\newblock Low rank methods for a class of generalized {L}yapunov equations and
  related issues.
\newblock {\em Numer. Math.}, 124(3):441--470, 2013.

\bibitem{Benner2009ADI}
Peter Benner, Ren-Cang Li, and Ninoslav Truhar.
\newblock On the {ADI} method for {S}ylvester equations.
\newblock {\em J. Comput. Appl. Math.}, 233(4):1035--1045, 2009.

\bibitem{Beyn}
Wolf-J\"{u}rgen Beyn.
\newblock An integral method for solving nonlinear eigenvalue problems.
\newblock {\em Linear Algebra Appl.}, 436(10):3839--3863, 2012.

\bibitem{BBB-KhatriRao-First15}
David~J. Biagioni, Daniel Beylkin, and Gregory Beylkin.
\newblock Randomized interpolative decomposition of separated representations.
\newblock {\em J. Comput. Phys.}, 281:116--134, 2015.

\bibitem{boulton2007approximation}
Lyonell Boulton and Michael Levitin.
\newblock On approximation of the eigenvalues of perturbed periodic
  {S}chr\"{o}dinger operators.
\newblock {\em J. Phys. A}, 40(31):9319--9329, 2007.

\bibitem{BujanovicKressner21}
Zvonimir Bujanovi\'c and Daniel Kressner.
\newblock Norm and trace estimation with random rank-one vectors.
\newblock {\em SIAM J. Matrix Anal. Appl.}, 42(1):202--223, 2021.

\bibitem{cohen2015optimal_arxiv}
Michael~B Cohen, Jelani Nelson, and David~P Woodruff.
\newblock Optimal approximate matrix product in terms of stable rank.
\newblock {\em arXiv preprint arXiv:1507.02268}, 2015.

\bibitem{Cohen2016Optimal}
Michael~B. Cohen, Jelani Nelson, and David~P. Woodruff.
\newblock Optimal approximate matrix product in terms of stable rank.
\newblock In {\em 43rd {I}nternational {C}olloquium on {A}utomata, {L}anguages,
  and {P}rogramming}, volume~55 of {\em LIPIcs. Leibniz Int. Proc. Inform.},
  pages Art. No. 11, 14. Schloss Dagstuhl. Leibniz-Zent. Inform., Wadern, 2016.

\bibitem{Davis06}
T.~A. Davis.
\newblock {\em Direct methods for sparse linear systems}, volume~2 of {\em
  Fundamentals of Algorithms}.
\newblock Society for Industrial and Applied Mathematics (SIAM), Philadelphia,
  PA, 2006.

\bibitem{De2000multilinear}
Lieven De~Lathauwer, Bart De~Moor, and Joos Vandewalle.
\newblock A multilinear singular value decomposition.
\newblock {\em SIAM J. Matrix Anal. Appl.}, 21(4):1253--1278, 2000.

\bibitem{Duersch2018robust}
Jed~A. Duersch, Meiyue Shao, Chao Yang, and Ming Gu.
\newblock A robust and efficient implementation of {LOBPCG}.
\newblock {\em SIAM J. Sci. Comput.}, 40(5):C655--C676, 2018.

\bibitem{Golub2013}
G.~H. Golub and C.~F. Van~Loan.
\newblock {\em Matrix computations}.
\newblock Johns Hopkins Studies in the Mathematical Sciences. Johns Hopkins
  University Press, Baltimore, MD, fourth edition, 2013.

\bibitem{Guettel2015}
Stefan G\"{u}ttel, Eric Polizzi, Ping Tak~Peter Tang, and Gautier Viaud.
\newblock Zolotarev quadrature rules and load balancing for the {FEAST}
  eigensolver.
\newblock {\em SIAM J. Sci. Comput.}, 37(4):A2100--A2122, 2015.

\bibitem{halko2011finding}
N.~Halko, P.~G. Martinsson, and J.~A. Tropp.
\newblock Finding structure with randomness: probabilistic algorithms for
  constructing approximate matrix decompositions.
\newblock {\em SIAM Rev.}, 53(2):217--288, 2011.

\bibitem{Hetmaniuk2006Basis}
U.~Hetmaniuk and R.~Lehoucq.
\newblock Basis selection in {LOBPCG}.
\newblock {\em J. Comput. Phys.}, 218(1):324--332, 2006.

\bibitem{Knyazev2007Block}
A.~V. Knyazev, M.~E. Argentati, I.~Lashuk, and E.~E. Ovtchinnikov.
\newblock Block locally optimal preconditioned eigenvalue xolvers ({BLOPEX}) in
  hypre and {PETS}c.
\newblock {\em SIAM J. Sci. Comput.}, 29(5):2224--2239, 2007.

\bibitem{Knyazev2001Toward}
Andrew~V. Knyazev.
\newblock Toward the optimal preconditioned eigensolver: locally optimal block
  preconditioned conjugate gradient method.
\newblock {\em SIAM J. Sci. Comput.}, 23(2):517--541, 2001.
\newblock Copper Mountain Conference (2000).

\bibitem{Knyazev2003geometric}
Andrew~V. Knyazev and Klaus Neymeyr.
\newblock A geometric theory for preconditioned inverse iteration. {III}. {A}
  short and sharp convergence estimate for generalized eigenvalue problems.
\newblock {\em Linear Algebra Appl.}, 358:95--114, 2003.

\bibitem{Kolda2009}
T.~G. Kolda and B.~W. Bader.
\newblock Tensor decompositions and applications.
\newblock {\em SIAM Review}, 51(3):455--500, 2009.

\bibitem{KressnerPerisa17}
Daniel Kressner and Lana Peri\v{s}a.
\newblock Recompression of {H}adamard products of tensors in {T}ucker format.
\newblock {\em SIAM J. Sci. Comput.}, 39(5):A1879--A1902, 2017.

\bibitem{KressnerSteinlechnerUschmajew14}
Daniel Kressner, Michael Steinlechner, and Andr\'{e} Uschmajew.
\newblock Low-rank tensor methods with subspace correction for symmetric
  eigenvalue problems.
\newblock {\em SIAM J. Sci. Comput.}, 36(5):A2346--A2368, 2014.

\bibitem{kressner2011preconditioned}
Daniel Kressner and Christine Tobler.
\newblock Preconditioned low-rank methods for high-dimensional elliptic {PDE}
  eigenvalue problems.
\newblock {\em Comput. Methods Appl. Math.}, 11(3):363--381, 2011.

\bibitem{Kuerschner16}
Patrick Kürschner.
\newblock {\em Efficient Low-Rank Solution of Large-Scale Matrix Equations},
  volume~45 of {\em Forschungsberichte aus dem Max-Planck-Institut für Dynamik
  komplexer technischer Systeme}.
\newblock Shaker Verlag, Aachen, 2016.

\bibitem{latala1997estimation}
Rafa{\l} Lata{\l}a.
\newblock Estimation of moments of sums of independent real random variables.
\newblock {\em Ann. Probab.}, 25(3):1502--1513, 1997.

\bibitem{LedouxProbability2011}
Michel Ledoux and Michel Talagrand.
\newblock {\em Probability in {B}anach spaces}.
\newblock Classics in Mathematics. Springer-Verlag, Berlin, 2011.

\bibitem{MartinssonTropp20}
Per-Gunnar Martinsson and Joel~A. Tropp.
\newblock Randomized numerical linear algebra: foundations and algorithms.
\newblock {\em Acta Numer.}, 29:403--572, 2020.

\bibitem{MatouOnvariants2008}
Ji\v{r}\'{\i} Matou\v{s}ek.
\newblock On variants of the {J}ohnson-{L}indenstrauss lemma.
\newblock {\em Random Structures Algorithms}, 33(2):142--156, 2008.

\bibitem{meyer2023hutchinson}
Raphael~A. Meyer and Haim Avron.
\newblock Hutchinson's estimator is bad at {K}ronecker-trace-estimation.
\newblock {\em arXiv preprint arXiv:2309.04952}, 2023.

\bibitem{MiedlarFOCM23}
A.~Miedlar, E.~de~Sturler, and A.~K. Saibaba.
\newblock Randomized contour integral methods for eigenvalue problems.
\newblock Foundations of Computational Mathematics 2023 conference, 2023.

\bibitem{MurrayDemmel23}
Riley Murray, James Demmel, Michael~W. Mahoney, N.~Benjamin Erichson, Maksim
  Melnichenko, Osman~Asif Malik, Laura Grigori, Piotr Luszczek, Michał
  Dereziński, Miles~E. Lopes, Tianyu Liang, Hengrui Luo, and Jack Dongarra.
\newblock Randomized numerical linear algebra: {A} perspective on the field
  with an eye to software.
\newblock {\em arXiv preprint arXiv:2302.11474v2}, 2023.

\bibitem{Pagh13}
Rasmus Pagh.
\newblock Compressed matrix multiplication.
\newblock {\em ACM Trans. Comput. Theory}, 5(3):Art. 9, 17, 2013.

\bibitem{Palitta21}
Davide Palitta.
\newblock Matrix equation techniques for certain evolutionary partial
  differential equations.
\newblock {\em J. Sci. Comput.}, 87(3):Paper No. 99, 36, 2021.

\bibitem{FEAST-Polizzi}
Eric Polizzi.
\newblock Density-matrix-based algorithm for solving eigenvalue problems.
\newblock {\em Phys. Rev. B}, 79:115112, Mar 2009.

\bibitem{rakhshan2020tensorized}
Beheshteh Rakhshan and Guillaume Rabusseau.
\newblock Tensorized random projections.
\newblock In {\em Proceedings of the International Conference on Artificial
  Intelligence and Statistics}, pages 3306--3316. PMLR, 2020.

\bibitem{Sakurai-FirstPaper}
Tetsuya Sakurai and Hiroshi Sugiura.
\newblock A projection method for generalized eigenvalue problems using
  numerical integration.
\newblock {\em J. Comput. Appl. Math.}, 159(1):119--128, 2003.

\bibitem{Sarlos2006}
Tamas Sarlos.
\newblock Improved approximation algorithms for large matrices via random
  projections.
\newblock In {\em Proceedings of the 47th Annual {IEEE} {S}ymposium on
  {F}oundations of {C}omputer {S}cience ({FOCS})}, FOCS '06, page 143–152,
  USA, 2006. IEEE Computer Society.

\bibitem{sun2018tensor}
Yiming Sun, Yang Guo, Joel~A Tropp, and Madeleine Udell.
\newblock Tensor random projection for low memory dimension reduction.
\newblock In {\em Proceedings of the 32nd Conference on Neural Information
  Processing Systems}. PMLR, 2018.

\bibitem{FEAST-SubspaceIteration}
Ping Tak~Peter Tang and Eric Polizzi.
\newblock F{EAST} as a subspace iteration eigensolver accelerated by
  approximate spectral projection.
\newblock {\em SIAM J. Matrix Anal. Appl.}, 35(2):354--390, 2014.

\bibitem{TrefethenWeideman14}
Lloyd~N. Trefethen and J.~A.~C. Weideman.
\newblock The exponentially convergent trapezoidal rule.
\newblock {\em SIAM Rev.}, 56(3):385--458, 2014.

\bibitem{Tucker1966Some}
Ledyard~R. Tucker.
\newblock Some mathematical notes on three-mode factor analysis.
\newblock {\em Psychometrika}, 31:279--311, 1966.

\bibitem{VershyninBookChapter12}
Roman Vershynin.
\newblock Introduction to the non-asymptotic analysis of random matrices.
\newblock In Yonina~C. Eldar and GittaEditors Kutyniok, editors, {\em
  Compressed Sensing: Theory and Applications}, pages 210--268. Cambridge
  University Press, 2012.

\bibitem{Vershynin2018High-dimensional}
Roman Vershynin.
\newblock {\em High-dimensional probability}, volume~47 of {\em Cambridge
  Series in Statistical and Probabilistic Mathematics}.
\newblock Cambridge University Press, Cambridge, 2018.

\bibitem{Vershynin20}
Roman Vershynin.
\newblock Concentration inequalities for random tensors.
\newblock {\em Bernoulli}, 26(4):3139--3162, 2020.

\bibitem{Woodruff2014Sketching}
David~P. Woodruff.
\newblock Sketching as a tool for numerical linear algebra.
\newblock {\em Found. Trends Theor. Comput. Sci.}, 10(1-2):iv+157, 2014.

\bibitem{Zhou2019Cluster}
Ming Zhou and Klaus Neymeyr.
\newblock Cluster robust estimates for block gradient-type eigensolvers.
\newblock {\em Math. Comp.}, 88(320):2737--2765, 2019.

\end{thebibliography}
 
\end{document}